\documentclass{amsart}
\usepackage{amsmath}
\usepackage{amsfonts}

\newtheorem{theorem}{Theorem}[section]
\theoremstyle{plain}

\newtheorem{corollary}[theorem]{Corollary}

\newtheorem{definition}[theorem]{Definition}
\newtheorem{example}[theorem]{Example}

\newtheorem{lemma}[theorem]{Lemma}
\newtheorem{notation}[theorem]{Notation}
\newtheorem{problem}[theorem]{Problem}
\newtheorem{proposition}[theorem]{Proposition}
\newtheorem{remark}[theorem]{Remark}

\numberwithin{equation}{section}
\input{tcilatex}

\begin{document}
\title[Isomorphisms of non-commutative domains]{Isomorphisms of
non-commutative domain algebras}
\author{Alvaro Arias}
\address{Department of Mathematics\\
University of Denver\\
Denver CO 80208}
\email{aarias@math.du.edu}
\urladdr{http://www.math.du.edu/\symbol{126}aarias}
\author{Fr\'{e}d\'{e}ric Latr\'{e}moli\`{e}re}
\address{Department of Mathematics\\
University of Denver\\
Denver CO 80208}
\email{frederic@math.du.edu}
\urladdr{http://www.math.du.edu/\symbol{126}frederic}
\date{November 2008}

\begin{abstract}
Noncommutative domain algebras were introduced by Popescu as the
non-selfadjoint operator algebras generated by weighted shifts on the Full
Fock space. This paper uses results from several complex variables to
classify many noncommutative domain algebras, and it uses results from
operator theory to obtain new bounded domains in $\mathbb{C}^{n}$ with
non-compact automorphic group.
\end{abstract}

\maketitle

\section{Introduction}

\qquad Noncommutative domain algebras, introduced in \cite{Po-domain},
generalize the noncommutative disk algebras and are defined as norm closures
of the algebras generated by a family of weighted shifts on the Full Fock
space. This paper investigates the isomorphism problem for this class of
algebras. The fundamental tool we use is the theory of functions in several
complex variables in $\mathbb{C}^{n}$. Formally, we apply Cartan's Lemma 
\cite{Cartan} and Sunada's Theorem \cite{Sunada} to domains in $\mathbb{C}%
^{n}$ naturally associated with our noncommutative domain algebras to derive
a first classification result: if there exists an isometric isomorphism
between two noncommutative domain algebras whose dual map takes the zero
character to the zero character, then the noncommutative domain algebras are
related via a simple linear transformation of their generators. When $n=2,$
Thullen characterization of domains of $%
\mathbb{C}
^{2}$ with non-compact automorphic group \cite{Thullen}\ gives more
information.

An application of our result is that there are many non-isomorphic
noncommutative domain algebras. In addition, we characterize the disk
algebras among all noncommutative domain algebras, and we obtain interesting
examples of new domains with noncompact automorphic group in $\mathbb{C}^{n}$%
.

Our paper is structured as follows. We first present a survey of the field
of noncommutative domain algebras, which allows us to introduce our
notations as well. We then construct contravariant functors between
noncommutative domain algebras with completely contractive unital
homomorphisms and domains in $%
\mathbb{C}
^{m}$ with holomorphic maps and deduce from them a few classification
results. Last, we present several applications to important examples of
noncommutative domains.

In this paper, the set $\mathbb{N}\backslash \left\{ 0\right\} $ of nonzero
natural numbers will be denoted by $\mathbb{N}^{\ast }$. The identity map on
a Hilbert space $\mathcal{H}$ will be denoted\ by $1_{\mathcal{H}}$, or
simply $1$ when no confusion may arise. The algebra of all bounded linear
operators on a Hilbert space $\mathcal{H}$ will be denoted by $\mathcal{B}%
\left( \mathcal{H}\right) $.

We also will adopt the convention that if $\chi $ is an $n$-tuple of
elements in a set $E$ then for all $i=1,\ldots ,n$ we write $\chi _{i}$ for
its $i^{th}$ projection, i.e. $\chi =\left( \chi _{1},\ldots ,\chi
_{n}\right) $.

The domain algebras are represented by formal power series of
\textquotedblleft free\textquotedblright\ variables $X_{1},\ldots ,X_{n}$.
For example, the non-commutative disc algebra with $n$ variable is given by $%
X_{1}+\cdots +X_{n}.$ To illustrate some computations and results of the
paper we look at the noncommutative domains with free formal power series%
\begin{equation*}
f=X_{1}+X_{2}+X_{1}X_{2}\text{ \qquad and\qquad\ }g=X_{1}+X_{2}+\frac{1}{2}%
X_{1}X_{2}+\frac{1}{2}X_{2}X_{1}
\end{equation*}%
and we prove that they are not isomorphic. We consider these algebras as
test examples. Their power series are close to the power series of the
noncommutative disc algebra and they cannot be distinguished by their $1$%
-dimensional representations.

\section{Background}

This first section allows us to introduce the category of noncommutative
domain algebras. These algebras were introduced in \cite{Po-domain} and much
notation is required for a proper description, so this section is also
presented as a guide for the reader about these objects.

\subsection{The Full Hilbert Fock Space of a Hilbert Space}

We start by introducing notations which we will use all throughout this
paper. For any Hilbert space $\mathcal{H}$, the full Fock space $\mathcal{F}%
\left( \mathcal{H}\right) $ of $\mathcal{H}$ is the completion of:%
\begin{equation*}
\dbigoplus\limits_{k\in \mathbb{N}}\mathcal{H}^{\otimes k}=\mathbb{C}\oplus 
\mathcal{H}\oplus \left( \mathcal{H}\otimes \mathcal{H}\right) \oplus \left( 
\mathcal{H}\otimes \mathcal{H}\otimes \mathcal{H}\right) \oplus \cdots
\end{equation*}%
for the Hilbert norm associated to the inner product $\left\langle
.,.\right\rangle $ defined on elementary tensors by:%
\begin{equation*}
\left\langle \xi _{0}\otimes \cdots \otimes \xi _{m},\zeta _{0}\otimes
\ldots \otimes \zeta _{k}\right\rangle =\left\{ 
\begin{array}{c}
0\text{ if }m\not=k\text{,} \\ 
\dprod\limits_{j=0}^{n}\left\langle \xi _{j},\zeta _{j}\right\rangle _{%
\mathcal{H}}\text{ otherwise,}%
\end{array}%
\right.
\end{equation*}%
where $\left\langle .,.\right\rangle _{\mathcal{H}}$ is the inner product on 
$\mathcal{H}$. We shall denote $\mathcal{F}\left( \mathbb{C}^{n}\right) $ by 
$\mathcal{F}_{n}$ for all $n\in \mathbb{N}\backslash \left\{ 0\right\} $. We
now shall exhibit, for any nonzero natural number $n$, a natural isomorphism
of Hilbert space between $\mathcal{F}_{n}$ and $\ell ^{2}\left( \mathbb{F}%
_{n}^{+}\right) $, where $\mathbb{F}_{n}^{+}$ is the free semigroup on $n$
generators $g_{1},\ldots ,g_{n}$ with identity element $e$.

Let $n\in \mathbb{N}^{\ast }$ be fixed and let $\left\{ e_{1},\ldots
,e_{n}\right\} $ be the canonical basis of $\mathbb{C}^{n}$. For any word $%
\alpha \in \mathbb{F}_{n}^{+}$ we define:%
\begin{equation*}
e_{\alpha }=e_{i_{1}}\otimes e_{i_{2}}\otimes \cdots \otimes
e_{i_{\left\vert \alpha \right\vert }}
\end{equation*}%
where $\left\vert \alpha \right\vert $ is the length of $\alpha $ and $%
\alpha =g_{i_{1}}\ldots g_{i_{\left\vert \alpha \right\vert }}$. Note that
this decomposition of $\alpha $ is unique in the semigroup $\mathbb{F}%
_{n}^{+}$. The map which sends, for any $\alpha \in \mathbb{F}_{n}^{+}$, the
element $e_{\alpha }\in \mathcal{F}_{n}$ to the series $\delta _{\alpha }\in
\ell ^{2}\left( \mathbb{F}_{n}^{+}\right) $ defined by:%
\begin{equation*}
\forall \beta \in \mathbb{F}_{n}^{+}\ \ \ \ \delta _{\alpha }\left( \beta
\right) =\left\{ 
\begin{array}{c}
1\text{ if }\alpha =\beta \text{,} \\ 
0\text{ otherwise}%
\end{array}%
\right.
\end{equation*}%
extends to a linear isomorphism from $\mathcal{F}_{n}$ onto $\ell ^{2}\left( 
\mathbb{F}_{n}^{+}\right) $. The easy proof is left to the reader. In this
paper, we will identify these two Hilbert spaces.

\subsection{Left Creation Operators and Noncommutative Disk Algebras}

At the root of the study of disk algebras reside the concept of row
contractions:

\begin{definition}
Let $\mathcal{H}$ be a Hilbert space and $n\in \mathbb{N}^{\ast }$. An $n$%
-row contraction on $\mathcal{H}$ is an $n$-tuple of operators $\left(
T_{1},\ldots ,T_{n}\right) $ on $\mathcal{H}$ such that:%
\begin{equation*}
\sum_{i=0}^{n}T_{i}T_{i}^{\ast }\leq 1_{\mathcal{H}}\text{.}
\end{equation*}
\end{definition}

In other words, the operator $T=\left[ T_{1}\ T_{2}\ \cdots \ T_{n}\right] $
acting on $\mathcal{H}^{n}$ is a contraction since $TT^{\ast }\leq 1_{%
\mathcal{H}_{n}}$. We now construct a row contraction which plays a
fundamental role in our study:

\begin{definition}
Let $n\in \mathbb{N}^{\ast }$. For $i\in \left\{ 1,\ldots ,n\right\} $, the $%
i^{th}$ left creation operator $S_{i}$ on $\mathcal{F}_{n}$ is the unique
continuous linear extension of:%
\begin{equation*}
\forall \alpha \in \mathbb{F}_{n}^{+}\ \ \ \ S_{i}\left( \delta _{\alpha
}\right) =\delta _{g_{i}\alpha }\text{.}
\end{equation*}
\end{definition}

We note that the linear extension of $\delta _{\alpha }\mapsto \delta
_{g_{i}\alpha }$ is an isometry on the span of $\left\{ \delta _{\alpha
}:\alpha \in \mathbb{F}_{n}^{+}\right\} $ and therefore is uniformly
continuous on this span which is dense in $\mathcal{F}_{n}$ so the
continuous extension is well-defined and unique. One can easily check that

\begin{proposition}
The $n$-tuple $\left( S_{1},\ldots ,S_{n}\right) $ is a row contraction of
isometries with orthogonal ranges.
\end{proposition}

In this paper, we will focus our attention on algebras generated by various
shifts operators. We introduce:

\begin{definition}
The noncommutative disk algebra $\mathcal{A}_{n}$ is the norm closure of the
algebra generated by $\left\{ S_{1},\ldots ,S_{n},1\right\} $.
\end{definition}

This algebra enjoys the following remarkable property:

\begin{theorem}[G. Popescu \protect\cite{Po-vninequality}]
\label{Popescu1}Let $\left( T_{1},\ldots ,T_{n}\right) $ be a $n$-tuple of
operators on a Hilbert space $\mathcal{H}$. Then $\left( T_{1},\ldots
,T_{n}\right) $ is a row contraction if and only if there exists a
completely bounded unital morphism $\psi :\mathcal{A}_{n}\longrightarrow 
\mathcal{B}\left( \mathcal{H}\right) $ such that $\psi (S_{i})=T_{i}$ for $%
i=1,\ldots ,n$.
\end{theorem}

\bigskip We will sketch a proof of this result at the end of this section.
But first, we notice that Theorem (\ref{Popescu1})\ formalizes the
fundamental idea that $\left( S_{1},\ldots ,S_{n}\right) $ are universal
among all row contractions. We also notice the following useful corollary:

\begin{corollary}[G. Popescu \protect\cite{Po-proceedings}]
The spectrum of $\mathcal{A}_{n}$ (i.e. the set of its characters, or one
dimensional representations) is the closed unit ball of $\mathbb{C}^{n}$.
\end{corollary}

\begin{proof}
A character of $\mathcal{A}_{n}$ is fully defined by its image on the
generators $S_{1},\ldots ,S_{n}$. Moreover, any unital morphism from $%
\mathcal{A}_{n}$ into the Abelian algebra $\mathbb{C}$ is completely
contractive. Let $\left( z_{1},\ldots ,z_{n}\right) \in \mathbb{C}^{n}$.
Then by Theorem (\ref{Popescu1}) there exists a morphism $\psi $ from $%
\mathcal{A}_{n}$ into $\mathbb{C}$ with $\psi (S_{i})=z_{i}$ for all $%
i=1,\ldots ,n$ if and only if $\left[ z_{1}\ \cdots \ z_{n}\right] $ is a
row contraction, i.e. $1\geq \sum_{i=1}^{n}z_{i}\overline{z_{i}}%
=\sum_{i=1}^{n}\left\vert z_{i}\right\vert ^{2}$ as claimed.
\end{proof}

We now sketch a proof of Theorem (\ref{Popescu1}). The original proof of
Theorem (\ref{Popescu1}) combined a dilation theorem for row contractions
with a Wold decomposition of isometric row contractions. In \cite{Po-poisson}%
, Popescu found a shorter proof using Poisson transforms, which are explicit
dilations for row contractions $\left( T_{1},\ldots ,T_{n}\right) $ with the
additional property that, if $\varphi :A\in \mathcal{B}\left( \mathcal{H}%
\right) \mapsto \sum_{i=1}^{n}T_{i}AT_{i}^{\ast }$ then $\lim_{k\rightarrow
\infty }\varphi ^{k}(A)=0$ in the strong operator topology for all $A\in 
\mathcal{B}\left( \mathcal{H}\right) $. Such row contractions will be called 
$C_{\cdot 0}$-row contraction. The proof goes as follows: assume that $%
T=\left( T_{1},\ldots ,T_{n}\right) $ is a $C_{\cdot 0}$-row contraction on
a Hilbert space $\mathcal{H}$ and let $\Delta =\left( I-\sum_{i\leq
n}T_{i}T_{i}^{\ast }\right) ^{1/2}$. Define $K:\mathcal{H}\rightarrow \ell
_{2}\left( \mathbb{F}_{n}^{+}\right) \otimes \mathcal{H}$ by%
\begin{equation*}
\forall h\in \mathcal{H}\ \ \ \ K(h)=\sum_{\alpha \in \mathbb{F}%
_{n}^{+}}\delta _{a}\otimes \Delta \left( T_{\alpha }\right) ^{\ast }(h)%
\text{.}
\end{equation*}%
It is not hard to check that $K$ is an isometry that satisfies:%
\begin{equation*}
\forall i\in \left\{ 1,\ldots ,n\right\} \ \ \ \ K^{\ast }\left(
S_{i}\otimes 1\right) =T_{i}K^{\ast }\text{.}
\end{equation*}%
From this it follows that if we set:%
\begin{equation*}
\Phi :a\in \mathcal{A}_{n}\longmapsto K^{\ast }\left( a\otimes I\right) K
\end{equation*}%
then $\Phi $ is a completely contractive unital homomorphism which satisfies
Theorem (\ref{Popescu1}). The map $K\ $is called the Poisson kernel of $%
\left( T_{1},\ldots ,T_{n}\right) $, and the map $\Phi $ is called the
Poisson transform of $\left( T_{1},\ldots ,T_{n}\right) $.

The general case for row contractions follows from the observation that if $%
T=\left( T_{1},\ldots ,T_{n}\right) $ is a row contraction, then $%
T_{r}=\left( rT_{1},\ldots ,rT_{n}\right) $ is a $C_{\cdot 0}$-row
contraction for $0<r<1$. One then concludes by taking the limit as $%
r\rightarrow 1$.

The converse assertion in Theorem (\ref{Popescu1}) is clear.

\bigskip

The Poisson transform is a useful tool to study these algebras. For example
in \cite{Arias-Popescu-1} Poisson transforms were used to obtain
multivariate non-commutative and commutative Nevanlinna-Pick and
Caratheodory interpolation theorems. The same results were obtained in \cite%
{Davidson-Pitts-interpolation} using different methods. Multivariate
interpolation problems, particularly commutative ones, have received a lot
of attention in recent years.

This paper deals with natural generalizations of $\mathcal{A}_{n}$ where the
left creation operators are replaced by weighted shifts. We present these
notions in the next section.

\subsection{Weighted Shifts and Noncommutative Domain Algebras}

Popescu and the first author \cite{Arias-Popescu-2} proved that the
construction of a Poisson kernel and Poisson transform, as was done
previously for row-contractions, could be extended to a class of weighted
shifts $\left( T_{1},\ldots ,T_{n}\right) $ satisfying the condition that $%
\sum_{\alpha \in \mathbb{F}_{n}^{+}}a_{\alpha }T_{\alpha }T_{\alpha }^{\ast
}\leq 1$ for some coefficients $a_{\alpha }\geq 0$ ($\alpha \in \mathbb{F}%
_{n}^{+}$) in lieu of $\sum_{i=1}^{n}T_{i}T_{i}^{\ast }\leq 1$ and with the
notations defined below:

\begin{definition}
Let $\alpha \in \mathbb{F}_{n}^{+}$ and $T_{1},\ldots ,T_{n}$ be operators
on $\mathcal{H}$. Then $T_{\alpha }=T_{i_{1}}\ldots T_{i_{\left\vert \alpha
\right\vert }}$ where $\alpha =g_{i_{1}}\ldots g_{i_{\left\vert \alpha
\right\vert }}$. In other words, $\alpha \mapsto T_{\alpha }$ is the unique
unital multiplicative map such that $g_{i}\mapsto T_{i}$ for $i=1,\ldots ,n$.
\end{definition}

The coefficients $\left( a_{\alpha }\right) _{\alpha \in \mathbb{F}_{n}^{+}}$
were defined in relation with the weight of the shifts $\left( T_{1},\ldots
,T_{n}\right) $ as in a paper of Quiggin and were known to satisfy $%
a_{g_{i}}>0$ for all $i=1,\ldots ,n$ and $a_{\alpha }\geq 0$ for $\left\vert
\alpha \right\vert \geq 1$. They were in general difficult to compute.

\bigskip

In a more recent paper, Popescu \cite{Po-domain} generalizes further the
study of weighted shifts and their associated algebras by starting with a
general collection of nonnegative coefficients $\left( a_{\alpha }\right)
_{\alpha \in \mathbb{F}_{n}^{+}}$ satisfying the following conditions:

\begin{definition}
Let $\left( a_{\alpha }\right) _{\alpha \in \mathbb{F}_{n}^{+}}$ be a family
of nonnegative numbers and consider the formal power series in $n$ free
variables $X_{1},\ldots ,X_{n}$ defined by $f=\sum a_{\alpha }X_{\alpha }$
with $X_{a}=X_{i_{1}}X_{i_{2}}\cdots X_{i_{\left\vert \alpha \right\vert }}$
when $\alpha =g_{i_{1}}\ldots g_{i_{\left\vert \alpha \right\vert }}$. Then $%
f$ is called a positive regular $n$-free formal power series when the
following conditions are met:%
\begin{equation}
\left\{ 
\begin{array}{l}
a_{e}=0\text{,} \\ 
a_{g_{i}}>0\text{ for all }i\in \left\{ 1,\ldots ,n\right\} \text{,} \\ 
\sup_{n\in \mathbb{N}^{\ast }}\left( \left\vert \sum_{\left\vert \alpha
\right\vert =n}a_{\alpha }^{2}\right\vert \right) ^{\frac{1}{n}}=M<\infty 
\text{.}%
\end{array}%
\right.  \label{Condition1}
\end{equation}
\end{definition}

Popescu then produced a universal model for all operator $n$-tuples
satisfying $\sum_{\alpha \in \mathbb{F}_{n}^{+}}a_{\alpha }T_{\alpha
}T_{\alpha }^{\ast }\leq 1$ such that $f=\sum a_{\alpha }X_{\alpha }$ is a
positive regular $n$-free formal power series, based upon weighted shifts,
in a manner similar to the above constructions of the disk algebra. More
formally, he proves in \cite{Po-domain} the following fundamental result:

\begin{theorem}[G. Popescu \protect\cite{Po-domain}]
\label{Popescu2}Let us be given a positive regular $n$-free formal power
series $f=\sum a_{\alpha }X_{\alpha }$. Then there exists positive real
numbers $\left( b_{\alpha }\right) _{\alpha \in \mathbb{F}_{n}^{+}}$ such
that, if for all $i=1,\ldots ,n$ we define $W_{i}^{f}$ as the unique
continuous linear extension of:%
\begin{equation*}
W_{i}^{f}:\delta _{\alpha }\in \mathcal{F}_{n}\mapsto \sqrt[2]{\frac{%
b_{\alpha }}{b_{g_{i}\alpha }}}\delta _{g_{i}\alpha }
\end{equation*}%
then:%
\begin{equation*}
\sum_{\alpha \in \mathbb{F}_{n}^{+}}a_{\alpha }W_{\alpha }^{f}W_{\alpha
}^{f\ast }\leq 1\text{.}
\end{equation*}%
Moreover, if $\left( T_{1},\ldots ,T_{n}\right) $ is an $n$-tuple of
operators acting on some Hilbert space $\mathcal{H}$, then $\sum_{\alpha \in 
\mathbb{F}_{n}^{+}}a_{\alpha }T_{\alpha }T_{\alpha }^{\ast }\leq 1$ if and
only if there exists a unique completely contractive morphism of algebra $%
\Phi $ from the algebra generated\ by $\left\{ W_{1}^{f},\ldots
,W_{n}^{f}\right\} $ into $\mathcal{B}\left( \mathcal{H}\right) $ such that $%
W_{i}^{f}\mapsto T_{i}$ for all $i=1,\ldots ,n$.
\end{theorem}

The fundamental role played by the algebra generated by the weighted shifts $%
W_{1}^{f},\ldots ,W_{n}^{f}$ in Theorem (\ref{Popescu2}) justifies the
following definition, where we use the notations of Theorem (\ref{Popescu2}):

\begin{definition}
\label{NCDomainAlgDef}Let $f=\sum_{\alpha \in \mathbb{F}_{n}^{+}}a_{\alpha
}X_{\alpha }$ be a positive regular $n$-free formal power series. The
algebra $\mathcal{A}\left( \mathcal{D}_{f}\right) $ is the norm closure of
the algebra generated by the weighted shifts $W_{1}^{f},\ldots ,W_{n}^{f}$
of Theorem (\ref{Popescu2}) and the identity. The algebra $\mathcal{A}\left( 
\mathcal{D}_{f}\right) $ is called the noncommutative domain algebra
associated to $f$.
\end{definition}

In addition, following Popescu's notations, we define the noncommutative
domain associated to a positive regular free formal power series $f$ in a
Hilbert space $\mathcal{H}$ as the \textquotedblleft
preimage\textquotedblright\ of the unit ball by $f$ i.e.:

\begin{definition}
\label{NCDomain}Let $\mathcal{H}$ be a Hilbert space and $f=\sum_{\alpha \in 
\mathbb{F}_{n}^{+}}a_{\alpha }X_{\alpha }$ be a positive regular $n$-free
formal power series. Then:%
\begin{equation*}
\mathcal{D}_{f}\left( \mathcal{H}\right) =\left\{ \left( T_{1},\ldots
,T_{n}\right) \in \mathcal{B}\left( \mathcal{H}\right) :\sum_{\alpha \in 
\mathbb{F}_{n}^{+}}a_{\alpha }T_{\alpha }T_{\alpha }^{\ast }\leq 1\right\} 
\text{.}
\end{equation*}
\end{definition}

The following corollary of Theorem (\ref{Popescu2}) will play a fundamental
role in our paper, and its proof is simply a direct application of Theorem (%
\ref{Popescu2}).

\begin{corollary}
\label{corollary-representations}Let $T=\left( T_{1},\ldots ,T_{n}\right)
\in M_{k\times k}^{n}$. Then $T\in \mathcal{D}_{f}\left( \mathbb{C}%
^{k}\right) $ if and only if there exists a necessarily unique completely
contractive unital algebra morphism $\Phi :\mathcal{A}\left( \mathcal{D}%
_{f}\right) \longrightarrow M_{k\times k}$ such that $\Phi \left(
W_{i}^{f}\right) =T_{i}$ for $i\leq n.$
\end{corollary}

\begin{notation}
\label{notation}Sometimes we denote the $\Phi $ of Corollary \ref%
{corollary-representations} by $\left\langle T,\cdot \right\rangle _{f}^{n}$
and when no confusion may arise, we will denote $\left\langle T,\cdot
\right\rangle _{f}^{k}$ simply as $\left\langle T,\cdot \right\rangle $.
That is, 
\begin{equation*}
\left( \left\langle T,W_{1}^{f}\right\rangle _{f}^{k},\ldots ,\left\langle
T,W_{n}^{f}\right\rangle _{f}^{k}\right) =T\text{.}
\end{equation*}%
We also will write $\varphi (T)$ for $\left\langle T,\varphi \right\rangle $
for any $\varphi \in \mathcal{A}\left( \mathcal{D}_{f}\right) $ and $T\in 
\mathcal{D}_{f}\left( \mathbb{C}^{k}\right) $ following the notation in \cite%
{Po-domain} that emphasizes the functional calculus of $\mathcal{A}\left( 
\mathcal{D}_{f}\right) $ (see Lemma \ref{Holo}). The notation $\left\langle
.,.\right\rangle $ is meant to emphasize the role of duality in our present
work.
\end{notation}

\bigskip For the reader convenience, and as a mean to fix our notation, we
present a sktech of the proof of Theorem (\ref{Popescu2}). We refer the
reader to \cite{Po-domain} for the details of this proof and the development
of the general theory of noncommutative domain algebras.

\bigskip We fix a positive regular $n$-free formal power series $%
f=\sum_{\alpha \in \mathbb{F}_{n}^{+}}a_{\alpha }X_{\alpha }$. Let us choose 
$r>0$ such that $rM<\frac{1}{2}$. Denote by $\left\Vert .\right\Vert $ the
norm of operators acting on $\mathcal{F}_{n},$ the full Fock space. We have:%
\begin{eqnarray*}
\left\Vert \sum_{\left\vert \alpha \right\vert =k}a_{\alpha }r^{k}S_{\alpha
}\right\Vert ^{2} &=&r^{2k}\left\Vert \left( \sum_{\left\vert \alpha
\right\vert =k}a_{\alpha }r^{k}S_{\alpha }\right) ^{\ast }\left(
\sum_{\left\vert \alpha \right\vert =k}a_{\alpha }r^{k}S_{\alpha }\right)
\right\Vert \\
&=&r^{2k}\left\Vert \sum_{\left\vert \alpha \right\vert =\left\vert \beta
\right\vert =k}a_{\alpha }a_{\beta }S_{\alpha }^{\ast }S_{\beta }\right\Vert
\\
&=&r^{2k}\left\Vert \sum_{\left\vert \alpha \right\vert =k}a_{\alpha
}^{2}\right\Vert
\end{eqnarray*}%
so $\left\Vert \sum_{\left\vert \alpha \right\vert =k}a_{\alpha
}r^{k}S_{\alpha }\right\Vert \leq \left( rM\right) ^{k}$ and thus $%
\left\Vert \sum_{\left\vert \alpha \right\vert \geq 1}a_{\alpha
}r^{\left\vert \alpha \right\vert }S_{\alpha }\right\Vert <1$ and hence that 
$I-\sum_{\left\vert \alpha \right\vert \geq 1}a_{\alpha }r^{\left\vert
\alpha \right\vert }S_{\alpha }$ is invertible in $\mathcal{A}_{n}$.
Therefore there exist coefficients $\left( b_{\alpha }\right) _{\alpha \in 
\mathbb{F}_{n}^{+}}$ such that:%
\begin{equation}
\left( I-\sum_{\left\vert \alpha \right\vert \geq 1}a_{\alpha }r^{\left\vert
\alpha \right\vert }S_{\alpha }\right) ^{-1}=\sum_{\left\vert \alpha
\right\vert \geq 1}b_{\alpha }r^{\left\vert \alpha \right\vert }S_{\alpha
}\in \mathcal{A}_{n}\text{.}  \label{popescu2-1}
\end{equation}

The relation between $\left( b_{\alpha }\right) _{\alpha \in \mathbb{F}%
_{n}^{+}}$ and $\left( a_{\alpha }\right) _{\alpha \in \mathbb{F}_{n}^{+}}$
is given by the equalities:%
\begin{equation}
\left\{ 
\begin{tabular}{lll}
$\sum_{k=0}^{\infty }\left( \sum_{\left\vert \alpha \right\vert \geq
1}a_{\alpha }r^{\left\vert \alpha \right\vert }S_{\alpha }\right) ^{k}$ & $=$
& $\sum_{\left\vert \alpha \right\vert \geq 1}b_{\alpha }r^{\left\vert
\alpha \right\vert }S_{\alpha }$ \\ 
$\left( I-\sum_{\left\vert \alpha \right\vert \geq 1}a_{\alpha
}r^{\left\vert \alpha \right\vert }S_{\alpha }\right) \left(
\sum_{\left\vert \alpha \right\vert \geq 1}b_{\alpha }r^{\left\vert \alpha
\right\vert }S_{\alpha }\right) $ & $=$ & $1$ \\ 
$\left( \sum_{\left\vert \alpha \right\vert \geq 1}b_{\alpha }r^{\left\vert
\alpha \right\vert }S_{\alpha }\right) \left( I-\sum_{\left\vert \alpha
\right\vert \geq 1}a_{\alpha }r^{\left\vert \alpha \right\vert }S_{\alpha
}\right) $ & $=$ & $1$%
\end{tabular}%
\right.  \label{popescu2-2}
\end{equation}%
The first equality in (\ref{popescu2-2}) is valid because in any Banach
algebras $B$, if $t\in B$ and $\left\Vert t\right\Vert <1$ then $\sum_{k\geq
0}t^{k}=\left( 1-t\right) ^{-1}$. The other two are clear.

From the first equality in (\ref{popescu2-2}), we deduce that $b_{0}=1$, $%
b_{g_{i}}=a_{g_{i}}$ for $i=1,\ldots n$, $%
b_{g_{i}g_{j}}=a_{g_{i}g_{j}}+a_{g_{i}}a_{g_{j}},$ and more generally that
for $\alpha \in \mathbb{F}_{n}^{+}$:

\begin{equation}
b_{\alpha }=\sum_{j=1}^{\left\vert \alpha \right\vert }\sum_{\substack{ %
\gamma _{1}\cdots \gamma _{j}=\alpha  \\ \left\vert \gamma _{1}\right\vert
\geq 1,\ldots ,\left\vert \gamma _{j}\right\vert \geq 1}}a_{\gamma
_{1}}a_{\gamma _{2}}\cdots a_{\gamma _{j}}>0\text{.}  \label{popescu2-5}
\end{equation}%
This implies that 
\begin{equation}
b_{\alpha \beta }\geq b_{\alpha }b_{\beta }\text{.}  \label{popescu2-3}
\end{equation}

From the second and third equalities in (\ref{popescu2-2}) we deduce that
for $\left\vert \alpha \right\vert \geq 1$:%
\begin{equation}
b_{\alpha }=\sum_{\substack{ \beta \gamma =\alpha  \\ \left\vert \beta
\right\vert \geq 1}}a_{\beta }b_{\gamma }=\sum_{\substack{ \beta \gamma
=\alpha  \\ \left\vert \gamma \right\vert \geq 1}}b_{\beta }a_{\gamma }\text{%
.}  \label{popescu2-4}
\end{equation}

For $i\in \left\{ 1,\ldots ,n\right\} $ and $\alpha \in \mathbb{F}_{n}^{+}$
we define $W_{i}^{f}$ as the linear extension of:%
\begin{equation*}
W_{i}^{f}\delta _{\alpha }=\sqrt{\frac{b_{\alpha }}{b_{g_{i}\alpha }}}\delta
_{g_{i}\alpha }\text{.}
\end{equation*}%
It follows from Inequality (\ref{popescu2-3}) that for all $i=1,\ldots ,n$
we have $\sqrt{\frac{b_{\alpha }}{b_{g_{i}\alpha }}}\leq \sqrt{\frac{b_{0}}{%
b_{g_{i}}}}=\sqrt{\frac{1}{b_{g_{i}}}}$ and thus $\left\Vert
W_{i}^{f}\right\Vert =\sqrt{\frac{1}{b_{g_{i}}}}$. So $W_{i}^{f}$ can be
extended uniquely as a linear operator on $\mathcal{F}_{n}$.

\bigskip We now check that $\left( W_{1}^{f},\ldots ,W_{n}^{f}\right) \in 
\mathcal{D}_{f}\left( \mathcal{F}_{n}\right) $. Let $\alpha ,\beta \in 
\mathbb{F}_{n}^{+}$ with $\left\vert \beta \right\vert \geq 1$ we have:%
\begin{equation*}
W_{\alpha }^{f}W_{\alpha }^{f\ast }\delta _{\beta }=\left\{ 
\begin{array}{c}
\frac{b_{\gamma }}{b_{\beta }}\delta _{\beta }\text{ if }\beta =\gamma
\alpha \text{,} \\ 
0\text{ otherwise.}%
\end{array}%
\right.
\end{equation*}

\bigskip We deduce that $\left( \sum_{\left\vert \alpha \right\vert \geq
1}a_{\alpha }W_{\alpha }^{f}W_{\alpha }^{f\ast }\right) \delta _{\beta
}=\left( \sum_{\substack{ \alpha \gamma =\beta  \\ \left\vert \alpha
\right\vert \geq 1}}\frac{a_{\alpha }b_{\gamma }}{b_{\beta }}\right) \delta
_{\beta }=\delta _{\beta }$

Then if $\beta \in \mathbb{F}_{n}^{+}$ and $\left\vert \beta \right\vert
\geq 1$ and we have: 
\begin{equation*}
\left( \sum_{\left\vert \alpha \right\vert \geq 1}a_{\alpha }W_{\alpha
}^{f}W_{\alpha }^{f\ast }\right) \delta _{\beta }=\left( \sum_{\substack{ %
\alpha \gamma =\beta  \\ \left\vert \alpha \right\vert \geq 1}}\frac{%
a_{\alpha }b_{\gamma }}{b_{\beta }}\right) \delta _{\beta }=\delta _{\beta }
\end{equation*}%
by Equality (\ref{popescu2-4}). Then it follows that $\sum_{\left\vert
\alpha \right\vert \geq 1}a_{\alpha }W_{\alpha }^{f}W_{\alpha }^{f\ast }$ is
the projection orthogonal to $\mathbb{C}\delta _{0}$ and therefore is less
than or equal to the identity as desired.

\bigskip To conclude, we note that we can use a Poisson kernel again to
prove the characterization of completely bounded unital representations of $%
\mathcal{A}\left( \mathcal{D}_{f}\right) $. For any Hilbert space $\mathcal{H%
}$, any $\left( T_{1},\ldots ,T_{n}\right) \in \mathcal{D}_{f}\left( 
\mathcal{H}\right) $ satisfies the $C_{.0}$-condition when for all $A\in 
\mathcal{B}\left( \mathcal{H}\right) $ we have that $\lim_{k\rightarrow
\infty }\varphi ^{k}(A)=0$ in the strong operator topology with $\varphi
:A\in \mathcal{B}\left( \mathcal{H}\right) \mapsto \sum_{\alpha \in \mathbb{F%
}_{n}^{+}}a_{\alpha }T_{\alpha }AT_{\alpha }^{\ast }$. Given an element $%
\left( T_{1},\ldots ,T_{n}\right) $ of $\mathcal{D}_{f}\left( \mathcal{H}%
\right) $ satisfying the $C_{\cdot 0}$-condition, we define $\Delta =\sqrt[2]%
{1-\sum_{\left\vert \alpha \right\vert \geq 1}a_{\alpha }T_{\alpha
}T_{\alpha }^{\ast }}$ and set:%
\begin{equation*}
K:h\in \mathcal{H}\mapsto \sum_{\alpha \in \mathbb{F}_{n}^{+}}\sqrt[2]{%
b_{\alpha }}\delta _{\alpha }\otimes \Delta (T_{\alpha })^{\ast }\left(
h\right) \in \mathcal{F}_{n}\otimes \mathcal{H}\text{.}
\end{equation*}%
Then $K$ is an isometry such that $K^{\ast }\left( W_{i}^{f}\otimes 1\right)
K=T_{i}$ for $i=1,\ldots ,n$. The computations which remain are along the
same line as for Theorem (\ref{Popescu1}).

\bigskip We now set informally the problem which this paper begin to address:

\begin{problem}
Given positive regular free formal power series $f$ and $g,$ under what
conditions is $\mathcal{A}\left( \mathcal{D}_{f}\right) $ isomorphic to $%
\mathcal{A}\left( \mathcal{D}_{g}\right) $?
\end{problem}

\bigskip To be more precise, we introduce the proper notion of isomorphism
for this paper:

\begin{definition}
Let \textrm{NCD} be the category of noncommutative domains, defined as
follows: the objects of \textrm{NCD} are the noncommutative domain algebras $%
\mathcal{A}\left( \mathcal{D}_{f}\right) $ for all positive regular free
formal power series $f$ and the morphisms between these algebras are the
completely isometric algebra unital isomorphisms.
\end{definition}

Given a Hilbert space $\mathcal{H}$, Definition (\ref{NCDomain})\ shows how
to associate to any object in \textrm{NCD} a set of operators. We will see
that when $\mathcal{H}$ is finite dimensional then we can extends this map
between objects into a functor of well-chosen categories. This functor will
be the invariant we shall use further on to distinguish between many objects
in \textrm{NCD}.

\section{Isomorphisms between noncommutative domains}

We start by observing that a completely bounded isomorphism between
noncommutative domain algebras gives rise to a sequence of multivariate
holomorphic maps. We thus construct a functor from \textrm{NCD} to the
category of domains in $%
\mathbb{C}
^{k}$. We start by recalling basic definitions \cite{Krantz-book} from
multivariate analysis to set our notations.

A $n$-index is an element $k\in \mathbb{N}^{n}$. Given $z=\left(
z_{1},\ldots ,z_{n}\right) \in \mathbb{C}^{n}$ and $y=\left( y_{1},\ldots
,y_{n}\right) \in \mathbb{C}^{n}$, we define $x\cdot y=\left(
x_{1}y_{1},\ldots ,x_{n}y_{n}\right) $. In the same spirit, for any $n$%
-index $k=\left( k_{1},\ldots ,k_{n}\right) $, we write $z^{k}$ for $\left(
z_{1}^{k_{1}},\ldots ,z_{n}^{k_{n}}\right) $.

\begin{definition}
Let $U\subseteq \mathbb{C}^{n}$ be a nonempty open set. A function $%
f:U\longrightarrow \mathbb{C}^{n}$ is holomorphic on $U$ when, for all $z\in
U$, there exists a neighborhood $V\subseteq U$ of $z$ and complex numbers $%
\left( \alpha _{k}\right) _{k\in \mathbb{N}^{\ast k}}$ such that for all $%
y\in V$ we have:%
\begin{equation*}
f(y)=\sum \alpha _{k}\cdot \left( y-z\right) ^{k}\text{.}
\end{equation*}
\end{definition}

Note in particular that if we write $f=\left( f_{1},\ldots ,f_{n}\right) $
then $f_{i}$ is a function from $U$ to $\mathbb{C}$ which can be written as
a power series in its $n$ variables. Conversely, if $f_{1},\ldots ,f_{n}$
all are power series in their $n$ variables then $f$ is holomorphic.

To ease the discussion in this paper, we will allow ourselves the following
mild extension of the definition of holomorphy to sets with nonempty
interior:

\begin{definition}
Let $F\subseteq \mathbb{C}^{n}$ be a set of nonempty interior $\overset{%
\circ }{F}$. Then $f:F\longrightarrow \mathbb{C}^{n}$ is holomorphic on $F$
is $f$ restricted to $\overset{\circ }{F}$ is holomorphic on $\overset{\circ 
}{F}$ and $f$ is continuous on $F$.
\end{definition}

We will be interested in the matter of mapping a set onto another one by
means of a holomorphic map with a holomorphic inverse.

\begin{definition}
Let $U,V\subseteq \mathbb{C}^{n}$ be two sets with nonempty interiors. If
there exists a holomorphic map $f:U\longrightarrow V$ and a holomorphic map $%
g:V\longrightarrow U$ such that $f\circ g=\limfunc{Id}_{U}$ and $g\circ f=%
\limfunc{Id}_{V}$ then $U$ and $V$ are biholomorphically equivalent, and $f$
and $g$ are called biholomorphic maps.
\end{definition}

We refer to \cite{isaev-krantz} for a survey of the problem of holomorphic
mapping, which was a great motivation for the development of multivariate
complex analysis. Unlike the case of domains of the complex plane, where the
Riemann mapping theorem shows that any simply connected proper open subset
of $\mathbb{C}$ is bi-holomorphically equivalent to the open unit disk,
there are many biholomorphic equivalence classes of domains in $\mathbb{C}%
^{n}$ for $n>1$; in other words there is great rigidity for biholomorphic
maps between domains in $\mathbb{C}^{n}$ ($n>1$). We shall exploit this
rigidity in this paper to help classify the noncommutative domain algebras.

\subsection{Holomorphic Maps from Isomorphisms}

The fundamental observation of this paper is the following result which
creates a bridge between noncommutative domain algebra theory and the theory
of holomorphic mapping of domains in $\mathbb{C}^{m}$. To be formal, we set:

\begin{definition}
Let $k\in \mathbb{N}$. Let \textrm{HD}$_{k}$\textrm{\ }be the category of
domains in $\mathbb{C}^{k}$, i.e. the category whose objects are open
connected subsets of $\mathbb{C}^{k}$ for any $k\in \mathbb{N}$, and whose
morphisms are holomorphic maps.
\end{definition}

We shall now construct a contravariant functor from \textrm{NCD }to \textrm{%
HD}$_{k}$. To do so, we shall find the following two lemmas useful. These
two lemmas are part of Popescu's construction found in \cite{Po-domain}. We
include a proof of these results for the reader's convenience, as they can
be established directly, and it is again helpful to fix notations for the
rest of this paper.

The first result shows that any element in $\mathcal{A}\left( \mathcal{D}%
_{f}\right) $ admits a form of Fourier series expansion, where convergence
is understood in the radial sense. The second lemma proves that the pairings
between noncommutative domain algebras and the complex domains introduced in
(\ref{NCDomain}) define holomorphic functions.

\begin{lemma}[Popescu, \protect\cite{Po-domain}]
\label{radialCv}Let $a\in \mathcal{A}\left( \mathcal{D}_{f}\right) $ for
some positive regular free formal series $f$. Then there exists a unique
collection $\left( c_{\alpha }\right) _{\alpha \in \mathbb{F}_{n}^{+}}$ of
complex numbers such that for all $r\in \left( 0,1\right) $ the series:%
\begin{equation*}
\sum_{j=0}^{\infty }\sum_{\substack{ \alpha \in \mathbb{F}_{n}^{+}  \\ %
\left\vert \alpha \right\vert =j}}c_{\alpha }r^{\left\vert \alpha
\right\vert }W_{\alpha }
\end{equation*}%
converge in norm in $\mathcal{A}\left( \mathcal{D}_{f}\right) $ to $a_{r}$
defined by $\left\langle \left( rW_{1},\ldots ,rW_{n}\right) ,a\right\rangle
_{\ell ^{2}\left( \mathbb{F}_{n}^{+}\right) }$. Note that $%
\lim_{r\rightarrow 1^{-}}a_{r}=a$ in norm in $\mathcal{A}\left( \mathcal{D}%
_{f}\right) $.

We will denote $\sum_{\substack{ \alpha \in \mathbb{F}_{n}^{+}  \\ %
\left\vert \alpha \right\vert =j}}c_{\alpha }W_{\alpha }$ by $\left[ a\right]
_{j}$. Then $\left\Vert \left[ a\right] _{j}\right\Vert \leq \left\Vert
a\right\Vert $ in $\mathcal{A}\left( \mathcal{D}_{f}\right) $.
\end{lemma}

\begin{proof}
Let $a\in \mathcal{A}\left( \mathcal{D}_{f}\right) $ be fixed. By Definition
(\ref{NCDomainAlgDef}), the bounded linear operator $a$ acts on the Fock
space $\ell ^{2}\left( \mathbb{F}_{n}^{+}\right) $. In particular, there
exists a collection $\left( c_{\alpha }\right) _{\alpha \in \mathbb{F}%
_{n}^{+}}$ of complex numbers such that:%
\begin{equation*}
a\left( \delta _{0}\right) =\sum_{\alpha \in \mathbb{F}_{n}^{+}}c_{\alpha }%
\sqrt[2]{b_{\alpha }}\delta _{\alpha }
\end{equation*}%
where $\left\{ \delta _{\alpha }:\alpha \in \mathbb{F}_{n}^{+}\right\} $ is
the usual Hilbert basis of $\ell ^{2}\left( \mathbb{F}_{n}^{+}\right) $. We
now wish to prove that $a$ is the radial limit of $\sum_{j\in \mathbb{N}}%
\left[ a\right] _{j}$ in the sense described above. A standard Fourier-type
argument shows that we have for all $j\in \mathbb{N}$:%
\begin{equation}
\left\Vert \left[ a\right] _{j}\right\Vert \leq \left\Vert a\right\Vert 
\text{.}  \label{HoloIneq1}
\end{equation}%
Note that Inequality (\ref{HoloIneq1}) also follows from Lemma (\ref%
{NormLemma2}) later in this paper, and that lemma is independent from the
one we are now proving.

We conclude from Inequality (\ref{HoloIneq1}) that if $r\in (0,1)$ then $%
\sum_{j\in \mathbb{N}}r^{j}\left[ a\right] _{j}$ converges in norm in $%
\mathcal{A}\left( \mathcal{D}_{f}\right) $, and we denote this limit $a_{r}$%
. Using Theorem (\ref{Popescu2}), we also see that we must have:%
\begin{equation*}
a_{r}=\left\langle \left( rW_{1},\ldots ,rW_{n}\right) ,a\right\rangle \text{%
.}
\end{equation*}%
In particular, we note that $\left\Vert a_{r}\right\Vert \leq \left\Vert
a\right\Vert $ for all $r\in \left( 0,1\right) $. Last, if $a$ is a finite
linear combination of the operators $W_{\alpha }^{f}$ ($\alpha \in \mathbb{F}%
_{n}^{+}$) then clearly $\lim_{r\rightarrow 1^{-}}a_{r}=a$. In general, for
any $a$ and any $\varepsilon >0$ we can find a finite linear combination $p$
of the operators $W_{\alpha }^{f}$ ($\alpha \in \mathbb{F}_{n}^{+}$) such
that $\left\Vert a-p\right\Vert <\frac{1}{3}\varepsilon $. It is then
immediate that $\left\Vert a_{r}-p_{r}\right\Vert <\frac{1}{3}\varepsilon $
as well. Moreover, using continuity in $r$ there exists some $R\in \left(
0,1\right) $ such that $\left\Vert p-p_{r}\right\Vert <\frac{1}{3}%
\varepsilon $ for $r\in \left( R,1\right) $. Hence $\lim_{r\rightarrow
1^{-}}\left\Vert a-a_{r}\right\Vert =0$ as desired.
\end{proof}

As an observation and using the notation of Lemma (\ref{radialCv}), we note
that if $T$ is in $\mathcal{D}_{f}\left( \mathbb{C}^{k}\right) $ for some $%
k\in \mathbb{N}^{\ast }$ then the convergence of $a\left( T\right)
=\left\langle T,a\right\rangle =\sum_{j=0}^{\infty }\left( \sum_{\substack{ %
\alpha \in \mathbb{F}_{n}^{+}  \\ \left\vert \alpha \right\vert =j}}%
c_{\alpha }T_{\alpha }\right) $ is understood in the following sense:%
\begin{equation}
a\left( T_{1},\ldots ,T_{n}\right) =\left\langle T,a\right\rangle
=\lim_{r\rightarrow 1^{-}}\sum_{j=0}^{\infty }r^{j}\left( \sum_{\substack{ %
\alpha \in \mathbb{F}_{n}^{+}  \\ \left\vert \alpha \right\vert =j}}%
c_{\alpha }T_{\alpha }\right) \text{.}  \label{radial-convergence-of-T}
\end{equation}

In general, we will write that $a\in \mathcal{A}\left( \mathcal{D}%
_{f}\right) $ is the sum of the series $\left( \sum c_{\alpha }W_{\alpha
}\right) _{\alpha \in \mathbb{F}_{n}^{+}}$, or simply $a=\sum_{\alpha \in 
\mathbb{F}_{n}^{+}}c_{\alpha }W_{\alpha }$, where convergence is understood
as in Lemma (\ref{radialCv}).

\begin{lemma}[Popescu, \protect\cite{Po-domain}]
\label{Holo}Let $a\in \mathcal{A}\left( \mathcal{D}_{f}\right) $ for some
positive regular free formal series $f$. Then for all $k\in \mathbb{N}^{\ast
}$ the function:%
\begin{equation*}
T\in \mathcal{D}_{f}\left( \mathbb{C}^{k}\right) \mapsto a\left( T\right)
=\left\langle T,a\right\rangle _{k}\in M_{k\times k}
\end{equation*}%
is continuous on $\mathcal{D}_{f}\left( \mathbb{C}^{k}\right) $ and
holomorphic on the interior of $\mathcal{D}_{f}\left( \mathbb{C}^{k}\right) $%
.
\end{lemma}

\begin{proof}
Fix $a\in \mathcal{A}\left( \mathcal{D}_{f}\right) $ and $k\in \mathbb{N}%
^{\ast }$, and endow $\mathbb{C}^{nk^{2}}$ and $\mathbb{C}^{k^{2}}$ with an
arbitrary norm. The continuity of $\left\langle \cdot ,a\right\rangle _{k}$
is addressed first. Let $\varepsilon >0$ and $T\in \mathcal{D}_{f}\left( 
\mathbb{C}^{k}\right) $. Let $p\in \mathcal{A}\left( \mathcal{D}_{f}\right) $
be a finite linear combinations of the operators $W_{\alpha }$ for some
finite subset of indices $\alpha $ in $\mathbb{F}_{n}^{+}$ such that $%
\left\Vert a-p\right\Vert <\frac{1}{3}\varepsilon $. The map $\left\langle
\cdot ,p\right\rangle $, as seen as a function from $\mathbb{C}^{nk^{2}}$
into $\mathbb{C}^{k^{2}}$, is a $k^{2}$-tuple of polynomials in $nk^{2}$
variables, and thus is continuous on $\mathcal{D}_{f}\left( \mathbb{C}%
^{k}\right) $. Therefore there exists $\delta >0$ such that if $T^{\prime
}\in \mathcal{D}_{f}\left( \mathbb{C}^{k}\right) $ and $\left\Vert
T-T^{\prime }\right\Vert <\delta $ then $\left\Vert p\left( T\right)
-p\left( T^{\prime }\right) \right\Vert <\frac{1}{3}\varepsilon $. We
conclude:%
\begin{eqnarray*}
\left\Vert a\left( T\right) -a\left( T^{\prime }\right) \right\Vert &\leq
&\left\Vert a\left( T\right) -p\left( T\right) \right\Vert +\left\Vert
p\left( T\right) -p\left( T^{\prime }\right) \right\Vert +\left\Vert p\left(
T^{\prime }\right) -a\left( T^{\prime }\right) \right\Vert \\
&<&\frac{1}{3}\left( \varepsilon +\varepsilon +\varepsilon \right)
=\varepsilon \text{.}
\end{eqnarray*}%
This concludes the proof of continuity of $\left\langle \cdot
,a\right\rangle _{k}$ on $\mathcal{D}_{f}\left( \mathbb{C}^{k}\right) $.

We now turn to the question of holomorphy. Let $K$ be an arbitrary compact
subset of the interior of $\mathcal{D}_{f}\left( \mathbb{C}^{k}\right) $.
Let $\tau >0$ be chosen so that $\left( 1+\tau \right) K$ is also in the
interior of $\mathcal{D}_{f}\left( \mathbb{C}^{k}\right) $. Then, using
Lemma (\ref{radialCv}), we have for every $T\in K$ and $j\in \mathbb{N}$
that since $\left\langle \left( 1+\tau \right) T,\left[ a\right]
_{j}\right\rangle =\left( 1+\tau \right) ^{j}\left\langle T,\left[ a\right]
_{j}\right\rangle $ by definition:%
\begin{eqnarray}
\left\Vert \left\langle T,\left[ a\right] _{j}\right\rangle \right\Vert &=&%
\frac{1}{\left( 1+\tau \right) ^{j}}\left\Vert \left( 1+\tau \right)
^{j}\left\langle T,\left[ a\right] _{j}\right\rangle \right\Vert
\label{inquality-interior} \\
&=&\frac{1}{\left( 1+\tau \right) ^{j}}\left\Vert \left\langle \left( 1+\tau
\right) T,\left[ a\right] _{j}\right\rangle \right\Vert  \notag \\
&\leq &\frac{1}{\left( 1+\tau \right) ^{j}}\left\Vert \left[ a\right]
_{j}\right\Vert \leq \frac{1}{\left( 1+\tau \right) ^{j}}\left\Vert
a\right\Vert \text{.}  \notag
\end{eqnarray}%
Hence for any $N\in \mathbb{N}$:%
\begin{equation*}
\left\Vert \left\langle T,a\right\rangle -\left\langle T,\sum_{j=0}^{N}\left[
a\right] _{j}\right\rangle \right\Vert \leq \left\Vert a\right\Vert
\sum_{j=N+1}^{\infty }\frac{1}{\left( 1+\tau \right) ^{j}}\underset{%
N\rightarrow \infty }{\longrightarrow }0\text{.}
\end{equation*}%
Hence $\left\langle \cdot ,a\right\rangle $ is the uniform limit of the
functions $p_{N}=\left\langle \cdot ,\sum_{j=0}^{N}\left[ a\right]
_{j}\right\rangle $. These functions are given as $k^{2}$ polynomials in $%
nk^{2}$ variables, and are thus holomorphic the interior of $\mathcal{D}%
_{f}\left( \mathbb{C}^{k}\right) $. Hence $\left\langle \cdot
,a\right\rangle $ is holomorphic on the interior of $\mathcal{D}_{f}\left( 
\mathbb{C}^{k}\right) $.
\end{proof}

\bigskip

Using Inequality (\ref{inquality-interior}) and the notation of Lemma (\ref%
{radialCv}), we note that if $T\ $is in the interior of $\mathcal{D}%
_{f}\left( \mathbb{C}^{k}\right) $ for some $k\in \mathbb{N}^{\ast }$ then
we in fact have the norm convergence in $%
\mathbb{C}
^{k^{2}}$:%
\begin{equation*}
a\left( T_{1},\ldots ,T_{n}\right) =\left\langle T,a\right\rangle
=\sum_{j=0}^{\infty }\left( \sum_{\substack{ \alpha \in \mathbb{F}_{n}^{+} 
\\ \left\vert \alpha \right\vert =j}}c_{\alpha }T_{\alpha }\right) \text{.}
\end{equation*}

\bigskip We are now ready to deduce our main theorem:

\begin{theorem}
\label{Functor}Let $f,g$ be two positive regular free formal power series,
and let $n$ be the number of indeterminates of $f$. Let $\Phi :\mathcal{A}%
\left( \mathcal{D}_{f}\right) \mapsto \mathcal{A}\left( \mathcal{D}%
_{g}\right) $ be an isomorphism in \textrm{NCD}. Then for all $k\in \mathbb{N%
}^{\ast }$ there exists a biholomorphic function $\widehat{\Phi }_{k}:%
\mathcal{D}_{g}\left( \mathbb{C}^{k}\right) \longrightarrow \mathcal{D}%
_{f}\left( \mathbb{C}^{k}\right) $ such that for any $T\in \mathcal{D}%
_{g}\left( \mathbb{C}^{k}\right) $ we have:%
\begin{equation*}
\left\langle T,\Phi \left( \cdot \right) \right\rangle _{g}^{k}=\left\langle 
\widehat{\Phi }_{k}\left( T\right) ,\cdot \right\rangle _{f}^{k}\text{.}
\end{equation*}%
for $i=1,\ldots ,n$.

Moreover, for all $k\in \mathbb{N}^{\ast }$ the function $\widehat{\Phi }%
_{k} $ maps the interior of $\mathcal{D}_{g}\left( \mathbb{C}^{k}\right) $
onto the interior of $\mathcal{D}_{f}\left( \mathbb{C}^{k}\right) $, where
the interior is computed in the topology of $\mathbb{C}^{nk^{2}}$.

For $k\in \mathbb{N}^{\ast }$ we can thus define a contravariant functor $%
\mathbb{D}^{k}$ from \textrm{NCD} into \textrm{HD}$_{nk^{2}}$ where, for any
object $\mathcal{A}\left( \mathcal{D}_{f}\right) $ and isomorphism $\Phi $
in \textrm{NCD} in $\mathbb{D}^{k}\left( \mathcal{A}\left( \mathcal{D}%
_{f}\right) \right) $ is the interior of $\mathcal{D}_{f}\left( \mathbb{C}%
^{k}\right) $ and $\mathbb{D}^{k}\left( \Phi \right) =\widehat{\Phi }_{k}$.
\end{theorem}

\begin{proof}
Let $k\in \mathbb{N}^{\ast }$ and let $T\in \mathcal{D}_{g}\left( \mathbb{C}%
^{k}\right) $. It follows from the proof of Theorem (\ref{Popescu2}) that:%
\begin{equation*}
\widehat{\Phi _{k}}(T)=\left( \sum_{\alpha \in \mathbb{F}_{n}^{+}}c_{1,%
\alpha }T_{\alpha },\ldots ,\sum_{\alpha \in \mathbb{F}_{n}^{+}}c_{n,\alpha
}T_{\alpha }\right)
\end{equation*}%
where the convergence of each entry is understood in the sense of (\ref%
{radial-convergence-of-T}), and where $\Phi (W_{i}^{f})=\sum_{\alpha \in 
\mathbb{F}_{n}^{+}}c_{i,\alpha }W_{\alpha }^{g}$ for $i\leq n$.

By Lemma (\ref{Holo}), the function $T\in \mathcal{D}_{g}\left( \mathbb{C}%
^{k}\right) \mapsto \left\langle T,\Phi _{k}\left( W_{i}^{f}\right)
\right\rangle $ is continuous on $\mathcal{D}_{g}\left( \mathbb{C}%
^{k}\right) $ and holomorphic on the interior of $\mathcal{D}_{g}\left( 
\mathbb{C}^{k}\right) $ for all $i=1,\ldots ,n$. Therefore, $\widehat{\Phi
_{k}}$ is continuous on $\mathcal{D}_{g}\left( \mathbb{C}^{k}\right) $ and
holomorphic on the interior of $\mathcal{D}_{g}\left( \mathbb{C}^{k}\right)
. $ Moreover, since $\Phi ^{-1}:\mathcal{A}\left( \mathcal{D}_{g}\right)
\rightarrow \mathcal{A}\left( \mathcal{D}_{f}\right) $ is also an
isomorphism in \textrm{NCD}, we deduce the same properties for $\widehat{%
\Phi _{k}^{-1}}:\mathcal{D}_{f}\left( \mathbb{C}^{k}\right) \rightarrow 
\mathcal{D}_{g}\left( \mathbb{C}^{k}\right) .$

Last, let $T\in \mathcal{D}_{g}\left( \mathbb{C}^{k}\right) $ be an interior
point of $\mathcal{D}_{g}\left( \mathbb{C}^{k}\right) $ in the topology of $%
\mathbb{C}^{k^{2}n}$ and let $U\subseteq \mathcal{D}_{g}\left( \mathbb{C}%
^{k}\right) $ be an open set in $\mathbb{C}^{nk^{2}}$ with $T\in U$. By
assumption, $\widehat{\Phi }_{k}$ is an homeomorphism from $\mathcal{D}%
_{g}\left( \mathbb{C}^{k}\right) $ onto $\mathcal{D}_{f}\left( \mathbb{C}%
^{k}\right) $ (in their relative topology). In particular, identifying $%
\mathbb{C}^{nk^{2}}$ with $\mathbb{R}^{2nk^{2}}$ (via the canonical linear
isomorphism), $\widehat{\Phi }_{k}$ restricted to $U\subseteq \mathbb{R}%
^{2nk^{2}}$ is a continuous injective $\mathbb{R}^{2nk^{2}}$-valued
function, so by the invariance of domain principle \cite[Theorem 16, page 199%
]{spa}, we conclude that $\widehat{\Phi }_{k}(U)$ is open in $\mathbb{R}%
^{2nk^{2}}$ since $U$ is open in $\mathbb{R}^{2nk^{2}}$. Hence since $%
\widehat{\Phi }_{k}(T)\in \widehat{\Phi }_{k}(U)\subseteq \mathcal{D}%
_{f}\left( \mathbb{C}^{k}\right) $, we conclude that $\widehat{\Phi }%
_{k}\left( T\right) $ is indeed an interior point of $\mathcal{D}_{f}\left( 
\mathbb{C}^{k}\right) $.

Last, if $\left( T_{1},\ldots ,T_{n}\right) \in \mathcal{D}_{f}\left( 
\mathbb{C}^{k}\right) $ then so does $\left( tT_{1},\ldots ,tT_{n}\right) $
for $t\in \lbrack 0,1]$ so every point in $\mathcal{D}_{f}\left( \mathbb{C}%
^{k}\right) $ is path connected to $0$. Hence $\mathcal{D}_{f}\left( \mathbb{%
C}^{k}\right) $ is connected. Therefore the interior $\mathbb{D}^{k}\left( 
\mathcal{A}\left( \mathcal{D}_{f}\right) \right) $ of $\mathcal{D}_{f}\left( 
\mathbb{C}^{k}\right) $ is an object in \textrm{HD}$_{nk^{2}}$. It is then
easy to check that $\mathbb{D}^{k}$ defines a contravariant functor.
\end{proof}

\begin{remark}
\label{Functor1}When $k=1$, we do not need to assume that $\Phi $ is
completely isometric. Indeed, if $\Phi $ is a continuous unital algebra
isomorphism, then for any $\chi \in \mathcal{D}_{g}\left( \mathbb{C}\right) $
the map $\left\langle \chi ,\cdot \right\rangle \circ \Phi $ is continuous
and scalar valued, hence completely contractive, so the proof of Theorem\ (%
\ref{Functor}) applies.
\end{remark}

\begin{remark}
It is worth noticing that if the interiors of $\mathcal{D}_{f}\left( \mathbb{%
C}\right) $ and $\mathcal{D}_{g}\left( \mathbb{C}\right) $ are not
biholomorphically equivalent, then we can already conclude that there is no
bounded isomorphism between $\mathcal{A}\left( \mathcal{D}_{f}\right) $ and $%
\mathcal{A}\left( \mathcal{D}_{g}\right) $ --- not just no completely
isometric isomorphism.
\end{remark}

We introduce a notation to ease the presentation in this paper. Given an
object $\mathcal{A}\left( \mathcal{D}_{f}\right) $ in \textrm{NCD} we denote
the domain $\mathbb{D}^{k}\left( \mathcal{A}\left( \mathcal{D}_{f}\right)
\right) $ simply by $\mathbb{D}_{f}^{k}$. It is also worth pointing out a
terminology conflict between the literature in complex analysis and the
literature on non selfadjoint operator algebras. A domain in complex
analysis usually refers to a connected open subset of some Hermitian space,
while it is common in operator theory to call $\mathcal{D}_{f}\left( 
\mathcal{H}\right) $ domains for any Hilbert space $\mathcal{H}$, even
though the sets $\mathcal{D}_{f}\left( \mathcal{H}\right) $ are closed and
connected. We hope that the definition of the functors $\mathbb{D}^{k}$
clarify this point and establishes the bridge not only formally, but from a
lexical point of view as well.

Thus, the functors $\mathbb{D}^{k}$ ($k\in \mathbb{N}^{\ast }$) define a new
family of invariants for the noncommutative domains. We are thus led to
study the geometry of the sets $\mathbb{D}_{f}^{k}$ (seen as objects in 
\textrm{HD}$_{nk^{2}}$)\ in order to classify objects in \textrm{NCD}.

\subsection{The invariant $\mathbb{D}^{1}$}

As observed in Theorem(\ref{Functor})\ and Remark (\ref{Functor1}), the
biholomorphic equivalence of $\mathbb{D}_{f}^{1}$ is an isomorphism
invariant for the noncommutative domain $\mathcal{A}\left( \mathcal{D}%
_{f}\right) $ for any positive regular free formal series $f$. We thus can
provide a first easy partial classification result for the algebras $%
\mathcal{A}\left( \mathcal{D}_{f}\right) $ by studying the geometry of $%
\mathbb{D}_{f}^{1}$. We start by noticing that they have many symmetries, in
the following sense:

\begin{definition}
Let $D\subseteq \mathbb{C}^{m}$ be an open connected set. Then $D$ is a
Reinhardt domain if for all $\omega \in \mathbb{T}^{m}$ and $z\in D$ we have 
$\omega \cdot z=\left( \omega _{1}z_{1},\ldots ,\omega _{m}z_{m}\right) \in
D $.
\end{definition}

\begin{proposition}
Let $f$ be a positive regular free series. Then $\mathbb{D}_{f}^{1}$ is a
Reinhardt domain.
\end{proposition}

\begin{proof}
Let $\chi \in \mathbb{D}_{f}^{1}$ and $\omega =\exp \left( i\theta \right)
\in \mathbb{T}$ (where $\theta \in \left[ 0,2\pi \right] $). Then by
definition $\sum_{\alpha \in \mathbb{F}_{n}^{+}}a_{\alpha }\left\vert \chi
_{\alpha }\right\vert ^{2}<1$ so:%
\begin{equation*}
\sum_{\alpha \in \mathbb{F}_{n}^{+}}a_{\alpha }\left\vert \omega
^{\left\vert \alpha \right\vert }\chi _{\alpha }\right\vert
^{2}=\sum_{\alpha \in \mathbb{F}_{n}^{+}}a_{\alpha }\left\vert \chi _{\alpha
}\right\vert ^{2}<1\text{.}
\end{equation*}
\end{proof}

Now, two Reinhardt domains containing $0$ are biholomorphic only under very
rigid conditions, as proven by Sunada in \cite{Sunada}:

\begin{theorem}[Sunada \protect\cite{Sunada}]
\label{Sunada}Let $D$ and $D^{\prime }$ be two Reinhardt domains containing $%
0$. Then $D$ and $D^{\prime }$ are biholomorphic if and only if there exists
a permutation $\sigma $ and scalars $\lambda _{1},\ldots ,\lambda _{n}$ such
that the map:%
\begin{equation*}
\left( z_{1},\ldots ,z_{n}\right) \mapsto \left( \lambda _{1}z_{\sigma
(1)},\ldots ,\lambda _{n}z_{\sigma (n)}\right)
\end{equation*}%
is a biholomorphic map.
\end{theorem}

We can now conclude by putting Theorem (\ref{Functor})\ and Theorem (\ref%
{Sunada})\ together:

\begin{theorem}
\label{SunadaApp}Let $f,g$ be positive regular free formal power series.\ If 
$\mathcal{A}\left( \mathcal{D}_{f}\right) $ and $\mathcal{A}\left( \mathcal{D%
}_{g}\right) $ are isomorphic in \textrm{NCD} then:

\begin{enumerate}
\item The series $f$ and $g$ have the same number of indeterminate, denoted
by $n$,

\item There exists a permutation $\sigma $ of the set $\left\{ 1,\ldots
,n\right\} $ and a function $\lambda :\left\{ 1,\ldots ,n\right\}
\longrightarrow \mathbb{C}$ such that the map:%
\begin{equation*}
\left( z_{1},\ldots ,z_{n}\right) \in \mathbb{C}^{n}\mapsto \left( \lambda
_{1}z_{\sigma (1)},\ldots ,\lambda _{n}z_{\sigma (n)}\right)
\end{equation*}%
induces an isomorphism from $\mathbb{D}_{f}^{1}$ onto $\mathbb{D}_{g}^{1}$.
\end{enumerate}
\end{theorem}

We can already prove that there are many different noncommutative domain
algebras (i.e. non isomorphic) using Theorem (\ref{SunadaApp}).

Moreover, we shall see in a later subsection that it is possible to
construct isomorphisms in \textrm{NCD }from simply rescaling weighted shifts
in a manner suggested by Theorem (\ref{SunadaApp}). However, it is not true
that isomorphisms in \textrm{HD}$_{n\cdot 1^{2}}$ can be lifted to
isomorphisms in \textrm{NCD} even when they are presented in as nice a form
as in Theorem (\ref{SunadaApp}). The following example is our main
illustration of the results of this paper, as it displays the role of the
higher invariants $\mathbb{D}^{k}$ ($k>1$) to capture some of the
noncommutative aspects of the classification problem in \textrm{NCD}:

\begin{example}
\label{MainEx}Let $f=X_{1}+X_{2}+X_{1}X_{2}$ and $g=X_{1}+X_{2}+\frac{1}{2}%
\left( X_{1}X_{2}+X_{2}X_{1}\right) $. Then by definition $\mathbb{D}%
_{f}^{1}=\mathbb{D}_{g}^{1}=\left\{ \left( \lambda _{1},\lambda _{2}\right)
:\left\vert \lambda _{1}\right\vert ^{2}+\left\vert \lambda _{2}\right\vert
^{2}+\left\vert \lambda _{1}\lambda _{2}\right\vert ^{2}<1\right\} $. Yet we
shall see in the next section that $\mathcal{A}\left( \mathcal{D}_{f}\right) 
$ and $\mathcal{A}\left( \mathcal{D}_{g}\right) $ are not isomorphic in 
\textrm{NCD}.
\end{example}

In other words, Theorem (\ref{SunadaApp}) does not have a converse. \bigskip

The higher invariants $\mathbb{D}^{k}$ ($k>1$) will be used to show that
unital completely isometric isomorphisms that map the zero character to the
zero character are very simple. To illustrate this suppose that $f$ and $g$
are free formal power series and that $\Phi :\mathcal{A}\left( \mathcal{D}%
_{f}\right) \rightarrow \mathcal{A}\left( \mathcal{D}_{g}\right) $ is a
unital isometric isomorphism with $\widehat{\Phi }_{1}\left( 0\right) =0.$
Then $\widehat{\Phi }_{1}:\mathbb{D}_{g}^{1}\rightarrow \mathbb{D}_{f}^{1}$
is given by%
\begin{equation*}
\widehat{\Phi }_{1}\left( \lambda _{1},\ldots ,\lambda _{n}\right) =\left(
\sum_{\left\vert \alpha \right\vert \geq 1}c_{1,\alpha }\lambda _{\alpha
},\ldots ,\sum_{\left\vert \alpha \right\vert \geq 1}c_{n,a}\lambda _{\alpha
}\right) ,
\end{equation*}%
where $\Phi \left( W_{i}^{f}\right) =\sum_{\left\vert \alpha \right\vert
\geq 1}c_{i,\alpha }W_{\alpha }^{g}$ for $i\leq n.$ Now, by Cartan's Lemma,
which appears below, $\widehat{\Phi }_{1}$ is linear and we conclude that
for $i\leq n,$ 
\begin{equation*}
\sum_{\left\vert \alpha \right\vert \geq 1}c_{i,\alpha }\lambda _{\alpha
}=c_{i,g_{1}}\lambda _{1}+\cdots +c_{i,g_{n}}\lambda _{n}.
\end{equation*}%
Most terms cancel. Look at the left hand side and notice that the term $%
\left( \lambda _{1}\right) ^{2}$ appears only once and has coefficient $%
c_{1,g_{1}g_{1}}.$ Hence $c_{1,g_{1}g_{1}}=0.$ The term $\lambda _{1}\lambda
_{2}$ appears twice and has coefficients $c_{i,g_{1}g_{2}}$ and $%
c_{i,g_{2}g_{1}}.$ Hence $c_{i,g_{1}g_{2}}+c_{i,g_{2}g_{1}}=0.$ The main
point of the next subsection is that the higher invariants $\mathbb{D}^{k}$ (%
$k>1$) can be used to separate $c_{i,g_{1}g_{2}}$ from $c_{i,g_{2}g_{1}},$
and more generally, to conclude that $c_{i,\alpha }=c_{i,\alpha }=0$
whenever $\left\vert \alpha \right\vert >1.$

\subsection{The invariants $\mathbb{D}^{k}$}

Let $f$ be a positive regular free series. The domains $\mathbb{D}_{f}^{k}$
are not usually Reinhardt domains for $k\in \mathbb{N}^{\ast }\backslash
\left\{ 1\right\} $, yet they still enjoy enough symmetry to make some very
useful observations:

\begin{definition}
Let $D\subseteq \mathbb{C}^{m}$ be an open connected set. Then $D$ is
circular if for all $\omega \in \mathbb{T}$ and $z\in D$ then $\omega
z=\left( \omega z_{1},\ldots ,\omega z_{m}\right) \in D$.
\end{definition}

\begin{proposition}
For any positive regular free series $f$ and any $k\in \mathbb{N}^{\ast }$
the domain $\mathbb{D}_{f}^{k}$ is circular.
\end{proposition}

A powerful tool to classify circular domains is Cartan's lemma which we now
quote for the convenience of the reader.

\begin{theorem}[Cartan's Lemma \protect\cite{Cartan}]
\label{Cartan}Let $m\in \mathbb{N}^{\ast }$. Let $D,D^{\prime }$ be bounded
circular domains in $\mathbb{C}^{m}$ containing $0$. Let $\Psi
:D\longrightarrow D^{\prime }$ be biholomorphic such that $\Psi (0)=0$. Then 
$\Psi $ is the restriction of a linear map of $\mathbb{C}^{m}$.
\end{theorem}

Therefore, if two noncommutative domains $\mathcal{A}\left( \mathcal{D}%
_{f}\right) $ and $\mathcal{A}\left( \mathcal{D}_{g}\right) $ are isomorphic
in \textrm{NCD} then for all $k\in \mathbb{N}^{\ast }$ the domains $\mathbb{D%
}_{f}^{k}$ and $\mathbb{D}_{g}^{k}$ are linearly isomorphic, provided the
dual map of the isomorphism maps the zero character to the zero character.
Using this, we can conclude the main result of this paper:

\begin{theorem}
\label{Linear}Let $f,g$ be positive regular free series. Let $n$ be the
number of variables of $f$. Then if there exists an isomorphism $\Phi :%
\mathcal{A}\left( \mathcal{D}_{f}\right) \longrightarrow \mathcal{A}\left( 
\mathcal{D}_{g}\right) $ in \textrm{NCD }such that $\widehat{\Phi }_{1}(0)=0$
then there exists an invertible scalar $n\times n$ matrix $M$ such that:%
\begin{equation*}
\left[ 
\begin{array}{c}
\Phi (W_{1}^{f}) \\ 
\Phi (W_{2}^{f}) \\ 
\vdots \\ 
\Phi (W_{n}^{f})%
\end{array}%
\right] =\left( M\otimes 1_{\ell ^{2}\left( \mathbb{F}_{n}^{+}\right)
}\right) \left[ 
\begin{array}{c}
W_{1}^{g} \\ 
W_{2}^{g} \\ 
\vdots \\ 
W_{n}^{g}%
\end{array}%
\right] \text{.}
\end{equation*}%
(i.e. for each $i\leq n,$ $\Phi (W_{i}^{f})$ is a linear combinations of the
set of weighted shifts $\left\{ W_{1}^{g},\ldots ,W_{n}^{g}\right\} $).
\end{theorem}

\begin{proof}
For $i\in \left\{ 1,\ldots ,n\right\} $ we write (see the proof of Theorem (%
\ref{Functor})):%
\begin{equation*}
\Phi \left( W_{i}^{f}\right) =\sum_{\alpha \in \mathbb{F}_{n}^{+}}c_{i,%
\alpha }W_{\alpha }^{g}\text{.}
\end{equation*}%
Let $k\in \mathbb{N}^{\ast }$ be fixed. By Theorem (\ref{Functor}), the maps 
$\widehat{\Phi }_{k}=\mathbb{D}_{k}\left( \Phi \right) $ are biholomorphic
maps from $\mathbb{D}_{g}^{k}$ onto $\mathbb{D}_{f}^{k}$, both being
circular domains containing $0$. Hence $\mathbb{D}_{g}^{1}$ is an open
subset of $\mathbb{C}^{n}$ and thus $g$ has as many variables as $f$. Since
for $\chi \in \mathbb{D}_{g}^{1},$ $(\chi _{1},\ldots ,\chi _{n})\in \mathbb{%
C}^{n}$ we have:%
\begin{equation*}
\widehat{\Phi }_{1}\left( \chi \right) =\left( \sum_{\alpha \in \mathbb{F}%
_{n}^{+}}c_{1,\alpha }\chi _{\alpha },\ldots ,\sum_{\alpha \in \mathbb{F}%
_{n}^{+}}c_{n,\alpha }\chi _{\alpha }\right)
\end{equation*}%
we conclude that $\widehat{\Phi }_{1}\left( 0\right) =0$ imposes that $%
c_{1,0}=\cdots =c_{n,0}=0$. Consequently, by construction $\widehat{\Phi }%
_{k}(0)=0$. Since for all $k\in \mathbb{N}^{\ast }$ the bounded domains $%
\mathbb{D}_{f}^{1}$ and $\mathbb{D}_{g}^{1}$ are circular, contain $0$ and
the maps $\widehat{\Phi }_{k}$ are biholomorphic, we conclude by Cartan's
Lemma (\ref{Cartan}) that $\widehat{\Phi }_{k}$ is the restriction of a
linear map. We now show that this implies that $\Phi $ is induced by a
linear map as stated in the theorem.

Fix $k\geq 2$ and $i\in \{1,\ldots ,n\}$. We shall exploit the fact that two
power series agree on an open set if and only if their coefficients agree.
Let $\beta \in \mathbb{F}_{n}^{+}$ be fixed and of length $k$; we write $%
\beta =g_{\beta _{1}}\ldots g_{\beta _{k}}$ where $\beta _{1},\ldots ,\beta
_{k}\in \left\{ 1,\ldots ,n\right\} $.

\medskip

Our goal is to prove that $c_{i,\beta }=0.$

\medskip

For any $j,l\in \left\{ 1,\ldots ,k\right\} $ and any matrix $N\in
M_{k\times k}\left( \mathbb{C}\right) $ we write $N_{j,l}$ for the $\left(
j,l\right) $ entry of $N$. Let $T=\left( T_{1},\ldots ,T_{n}\right) \in 
\mathbb{D}_{g}^{k},$ then viewing $\widehat{\Phi }_{k}\left( T\right) $ as
an $n$-tuple of matrices in $M_{k\times k}\left( \mathbb{C}\right) $, we let 
$\widehat{\Phi }_{k,i}\left( T\right) $ be the $i^{th}$ component of this $n$%
-tuple. We then set:%
\begin{equation*}
\eta =\left( \widehat{\Phi }_{k,i}\right) _{1,k}:\mathbb{D}%
_{g}^{k}\rightarrow 
\mathbb{C}%
\end{equation*}%
Then $\eta $ is a scalar valued holomorphic function, and in fact it has
expression:%
\begin{equation*}
\eta \left( T_{1},\ldots ,T_{n}\right) =\sum_{\alpha \in \mathbb{F}%
_{n}^{+}}c_{i,\alpha }\left( T_{\alpha }\right) _{1,k}.
\end{equation*}%
To ease notation, for $T=\left( T_{1},\ldots ,T_{n}\right) \in M_{k\times
k}\left( \mathbb{C}\right) ,$ $m\in \left\{ 1,\ldots ,n\right\} ,$ and $%
j,l\in \left\{ 1,\ldots ,k\right\} ,$ will denote $\left( T_{m}\right)
_{j,l} $ by $t_{jl}^{m}$ . Then $\eta $ is a holomorphic function on the
variables $t_{jl}^{m}$.

For a fixed $\alpha \in \mathbb{F}_{n}^{+}$ with $\left\vert \alpha
\right\vert =s$, the function $\left( T_{1},\ldots ,T_{n}\right) \mapsto
(T_{\alpha })_{1,k}$ is the sum of all homogeneous polynomials in the
entries of the matrices $T_{1},\ldots ,T_{n}$ of the form:%
\begin{eqnarray*}
\left( T_{1},\ldots ,T_{n}\right) &\mapsto &\left( T_{\alpha _{1}}\right)
_{1,i_{1}}\left( T_{\alpha _{2}}\right) _{i_{1},i_{2}}\cdots \left(
T_{\alpha _{s}}\right) _{i_{s-1},k} \\
&=&t_{1i_{1}}^{\alpha _{1}}\cdots t_{i_{s-1}k}^{\alpha _{s}}\text{.}
\end{eqnarray*}%
In particular, we note that $\eta $ is the sum of homogeneous polynomials,
and $\left( T_{\alpha }\right) _{1,k}$ is homogeneous of degree the length
of $\alpha $ for all $\alpha \in \mathbb{F}_{n}^{+}$. However, since $%
\widehat{\Phi }_{k}$ is linear, $\eta $ is also linear and therefore the sum
of all homogeneous polynomials of a degree greater than or equal to $2$
vanish.

We now wish to identity the coefficient of the polynomial $t_{11}^{\beta
_{1}}t_{12}^{\beta _{2}}t_{23}^{\beta _{3}}\cdots t_{(k-1)k}^{\beta _{k}}$
in $\eta $. Let:%
\begin{equation*}
\Sigma _{\beta }=\left\{ \alpha \in \mathbb{F}_{n}^{+}:\left( T_{\alpha
}\right) _{1,k}\text{ contains the term }t_{11}^{\beta _{1}}t_{12}^{\beta
_{2}}t_{23}^{\beta _{3}}\cdots t_{(k-1)k}^{\beta _{k}}\right\} \text{.}
\end{equation*}%
The coefficient of $t_{11}^{\beta _{1}}t_{12}^{\beta _{2}}t_{23}^{\beta
_{3}}\cdots t_{(k-1)k}^{\beta _{k}}$ in $\eta $ is equal to $\sum_{\alpha
\in \Sigma _{\beta }}c_{i,\alpha }$ and since $\eta $ is linear and $k\geq 2$
it follows that $\sum_{\alpha \in \Sigma _{\beta }}c_{i,\alpha }=0.$ To
finish the proof we will check that $\sum_{\beta }$ contains only one
element, which has to be $\beta .$ Then it will follow that $c_{i,\beta }=0$
as desired.

First, if $\alpha \in \Sigma _{\beta }$ then the length of $\alpha $ is the
length of $\beta $ i.e. $k$. Let us write $\alpha =g_{\alpha _{1}}\cdots
g_{\alpha _{k}}$. Now, $\left( T_{\alpha }\right) _{1,k}$ will be the sum of
the terms $t_{1i_{1}}^{\alpha _{1}}\cdots t_{i_{k-1}k}^{\alpha _{k}}$ for $%
i_{1},\ldots ,i_{k-1}\in \left\{ 1,\ldots ,n\right\} $. Hence there exists $%
i_{1},\ldots ,i_{k-1}\in \left\{ 1,\ldots ,k\right\} $ such that on $\mathbb{%
D}_{g}^{k}$ we have:%
\begin{equation}
t_{1i_{1}}^{\alpha _{1}}\cdots t_{i_{k-1}k}^{\alpha _{k}}=t_{11}^{\beta
_{1}}t_{12}^{\beta _{2}}t_{23}^{\beta _{3}}\cdots t_{(k-1)k}^{\beta _{k}}%
\text{.}  \label{LinearId1}
\end{equation}

Let us view the map $t_{jl}^{m}$ as defined from $\mathbb{C}^{nk^{2}}$ and
mapping the coordinate $mk^{2}+jk+l$, i.e. we linearly identify $M_{k\times
k}\left( \mathbb{C}\right) $ with $\mathbb{C}^{k^{2}}$ using canonical
bases. Since $\mathbb{D}_{g}^{k}$ is open in $\mathbb{C}^{nk^{2}}$\ and
contain $0$, Identity (\ref{LinearId1}) must hold on some open hypercube in $%
\mathbb{C}^{nk^{2}}$ around $0$. Hence the factors in Identity (\ref%
{LinearId1}) are identical up to some permutation. Now, there is a unique
factor of the form $t_{jk}^{m}$ on the right hand side, so there must be a
unique one on the left hand side. Moreover these two must agree. We deduce
that $i_{k-1}=k-1$ and $\alpha _{k}=\beta _{k}$. Now, there is a unique
factor of the form $t_{j\left( k-1\right) }^{m}$ on the right hand side, so
once again since $i_{k-2}=k-2$ and $\beta _{k-1}=\alpha _{k-1}$. By an easy
induction, we conclude that $\alpha _{m}=\beta _{m}$ for $m\in \left\{
1,\ldots ,k\right\} $. Consequently, $\Sigma _{\beta }=\left\{ \beta
\right\} .$

Since $\beta $ is arbitrary of length at least $2$ we conclude that $%
c_{i,\beta }=0$ for all $\beta $ of length at least $2$ and all $i\in
\left\{ 1,\ldots ,n\right\} $. Hence if we set:%
\begin{equation*}
M=\left[ 
\begin{array}{cccc}
c_{1,1} & c_{1,2} & \cdots & c_{1,n} \\ 
c_{2,1} & c_{2,2} & \cdots & c_{2,n} \\ 
\vdots & \vdots & \ddots & \vdots \\ 
c_{n,1} & c_{n,2} & \cdots & c_{n,n}%
\end{array}%
\right]
\end{equation*}%
it is now clear that $\Phi =\left( M\otimes 1_{\ell ^{2}\left( \mathbb{F}%
_{n}^{+}\right) }\right) $.
\end{proof}

\bigskip We now turn to some important examples of applications of Theorem (%
\ref{Linear}), including the important Example\ (\ref{MainEx}).

\section{Applications and Examples}

We now use Theorem (\ref{Linear}) to prove that various important examples
of pure formal power series give rise to non isomorphic algebras in \textrm{%
NCD}. First, we show that combining Thullen's classification of two
dimensional domains with Theorem (\ref{Linear}), we can distinguish in 
\textrm{NCD }the noncommutative domain algebras associated to $%
X_{1}+X_{2}+X_{1}X_{2}$ and $X_{1}+X_{2}+\frac{1}{2}\left(
X_{1}X_{2}+X_{2}X_{1}\right) $ which we encountered in Example\ (\ref{MainEx}%
). Note that the invariant $\mathbb{D}^{1}$ was not enough to do so, thus
showing the need to consider higher dimensional characters of noncommutative
domain algebras to capture the noncommutativity of the symbol defining them.

We then prove that the noncommutative disk algebras are exactly described by
degree one positive regular free series, thus allowing for a very simple
characterization of these algebras among all objects in \textrm{NCD}. In
particular, the algebras of Example\ (\ref{MainEx})\ are not disk algebras
--- and are, informally, the simplest such examples. In the process of our
investigation, we will also point out new examples of domains in $\mathbb{C}%
^{nk^{2}}$ with noncompact automorphism groups which occur naturally within
our framework.

\subsection{An Application of Thullen Characterization}

In 1931, Thullen \cite{Thullen} proved that if a bounded Reinhardt domain in 
$%
\mathbb{C}
^{2}$ has a biholomorphic map that does not map zero to zero, the domain is
linearly equivalent to one of the following:

\begin{enumerate}
\item Polydisc $\left\{ \left( z,w\right) \in 
\mathbb{C}
^{2}:\left\vert z\right\vert <1,\left\vert w\right\vert <1\right\} $,

\item Unit ball $\left\{ \left( z,w\right) \in 
\mathbb{C}
^{2}:\left\vert z\right\vert ^{2}+\left\vert w\right\vert ^{2}<1\right\} ,$

\item Thullen domain $\left\{ \left( z,w\right) \in 
\mathbb{C}
^{2}:\left\vert z\right\vert ^{2}+\left\vert w\right\vert ^{2/p}<1\right\} ,$
$p>0,p\neq 1.$
\end{enumerate}

We will use Thullen's Theorem to prove that the algebras in Example\ (\ref%
{MainEx})\ are non isomorphic in \textrm{NCD}. Indeed, an isomorphism $\Phi :%
\mathcal{A}\left( \mathcal{D}_{f}\right) \rightarrow \mathcal{A}\left( 
\mathcal{D}_{g}\right) $ induces a biholomorphic map $\widehat{\Phi }:%
\mathbb{D}_{g}^{1}\rightarrow \mathbb{D}_{f}^{1}.$ Since 
\begin{equation*}
\mathbb{D}_{f}^{1}=\mathbb{D}_{g}^{1}=\left\{ \left( \lambda _{1},\lambda
_{2}\right) \in 
\mathbb{C}
^{2}:\left\vert \lambda _{1}\right\vert ^{2}+\left\vert \lambda
_{2}\right\vert ^{2}+\left\vert \lambda _{1}\lambda _{2}\right\vert
^{2}<1\right\}
\end{equation*}
is not linearly equivalent to the three examples listed above, we conclude
that $\widehat{\Phi }\left( 0\right) =0.$

\bigskip

To proceed with the proof we list two useful lemmas, using the notations
established in Section 2:

\begin{lemma}
\label{NormLemma}For every $\alpha \in \mathbb{F}_{n}^{+}$ we have $%
\left\Vert W_{\alpha }\right\Vert =\sqrt[2]{\frac{1}{b_{\alpha }}}$.
\end{lemma}

\begin{proof}
Let $\alpha \in \mathbb{F}_{n}^{+}$. The operator $W_{\alpha }$ maps the
orthonormal basis $\left\{ \delta _{\beta }:\beta \in \mathbb{F}%
_{n}^{+}\right\} $ into an orthogonal family. Therefore:%
\begin{equation*}
\left\Vert W_{\alpha }\right\Vert =\sup_{\beta \in \mathbb{F}%
_{n}^{+}}\left\Vert W_{\alpha }\delta _{\beta }\right\Vert =\sup_{\beta \in 
\mathbb{F}_{n}^{+}}\sqrt[2]{\frac{b_{\beta }}{b_{\alpha \beta }}}\leq \sqrt[2%
]{\frac{b_{0}}{b_{\alpha }}}=\left\Vert W_{\alpha }\delta _{0}\right\Vert 
\text{.}
\end{equation*}%
Hence the lemma is proven.
\end{proof}

\begin{lemma}
\label{NormLemma2}For every $k\in \mathbb{N}$ and scalar family $\left(
x_{\alpha }\right) _{\alpha \in \mathbb{F}_{n}^{+}}$ we have:%
\begin{equation*}
\left\Vert \sum_{\left\vert \alpha \right\vert =k}x_{\alpha }W_{\alpha
}\right\Vert =\sqrt[2]{\sum_{\left\vert \alpha \right\vert =k}\frac{%
\left\vert x_{\alpha }\right\vert ^{2}}{b_{\alpha }}}\text{.}
\end{equation*}
\end{lemma}

\begin{proof}
Similarly to the proof of Lemma (\ref{NormLemma}), the operator $%
\sum_{\left\vert \alpha \right\vert =k}x_{\alpha }W_{\alpha }$ maps the
orthonormal family $\left\{ \delta _{\beta }:\beta \in \mathbb{F}%
_{n}^{+}\right\} $ onto an orthogonal family. Hence:%
\begin{eqnarray*}
\left\Vert \sum_{\left\vert \alpha \right\vert =k}x_{\alpha }W_{\alpha
}\right\Vert &=&\sup_{\beta \in \mathbb{F}_{n}^{+}}\left\Vert
\sum_{\left\vert \alpha \right\vert =k}x_{\alpha }W_{\alpha }\delta _{\beta
}\right\Vert =\sup_{\beta \in \mathbb{F}_{n}^{+}}\left\Vert \sum_{\left\vert
\alpha \right\vert =k}x_{\alpha }\sqrt[2]{\frac{b_{\beta }}{b_{\alpha \beta }%
}}\delta _{\alpha \beta }\right\Vert \\
&=&\sqrt[2]{\sum_{\left\vert \alpha \right\vert =k}\left\vert x_{\alpha
}\right\vert ^{2}\frac{b_{\beta }}{b_{\alpha \beta }}}\leq \sqrt[2]{%
\sum_{\left\vert \alpha \right\vert =k}\left\vert x_{\alpha }\right\vert ^{2}%
\frac{b_{0}}{b_{\alpha }}} \\
&=&\left\Vert \sum_{\left\vert \alpha \right\vert =k}x_{\alpha }W_{\alpha
}\delta _{0}\right\Vert \text{.}
\end{eqnarray*}%
This completes our proof.
\end{proof}

\bigskip We now can prove:

\begin{theorem}
\label{MainExThm}Let:%
\begin{equation*}
\left\{ 
\begin{tabular}{lll}
$f$ & $=$ & $X_{1}+X_{2}+X_{1}X_{2}$, \\ 
$g$ & $=$ & $X_{1}+X_{2}+\frac{1}{2}X_{1}X_{2}+\frac{1}{2}X_{2}X_{1}$.%
\end{tabular}%
\right.
\end{equation*}%
Then the noncommutative domain algebras $\mathcal{A}\left( \mathcal{D}%
_{f}\right) $ and $\mathcal{A}\left( \mathcal{D}_{g}\right) $ are not
isomorphic in \textrm{NCD}.
\end{theorem}

\begin{proof}
Remaining consistent with the notations we used in this paper, we will write 
$f=\sum_{\alpha \in \mathbb{F}_{n}^{+}}a_{\alpha }^{f}X_{\alpha }$ and $%
g=\sum_{\alpha \in \mathbb{F}_{n}^{+}}a_{\alpha }^{g}X_{\alpha }$ (so $%
a_{g_{1}}^{f}=a_{g_{2}}^{f}=a_{g_{1}}^{g}=a_{g_{2}}^{g}=1=a_{g_{1}g_{2}}^{f}$
yet $a_{g_{1}g_{2}}^{g}=a_{g_{2}g_{1}}^{g}=\frac{1}{2}$ and all other
coefficients are null). We also denote by $\left( b_{\alpha }^{f}\right)
_{\alpha \in \mathbb{F}_{n}^{+}}$ (resp. $\left( b_{\alpha }^{g}\right)
_{\alpha \in \mathbb{F}_{n}^{+}}$) the weight given by Equality (\ref%
{popescu2-5}) for the coefficients $\left( a_{\alpha }^{f}\right) _{\alpha
\in \mathbb{F}_{n}^{+}}$ (resp. $\left( a_{\alpha }^{g}\right) _{\alpha \in 
\mathbb{F}_{n}^{+}}$). The weighted shifts associated to $f$ (resp. $g$) are
denoted by $W_{\alpha }^{f}$ (resp. $W_{\alpha }^{g}$) for all $\alpha \in 
\mathbb{F}_{n}^{+}$.

Since, from Equality (\ref{popescu2-5}), we have $a_{g_{i}}=b_{g_{i}}$ and $%
b_{g_{i}g_{j}}=b_{g_{i}}a_{g_{i}}+a_{g_{i}g_{j}}$ ($i\not=j\in \left\{
1,2\right\} $) so we easily compute:%
\begin{equation}
\left\{ 
\begin{array}{l}
b_{g_{1}}^{f}=b_{g_{2}}^{f}=b_{g_{1}}^{g}=b_{g_{2}}^{g}=1\text{,} \\ 
b_{g_{1}g_{1}}^{f}=b_{g_{2}g_{1}}^{f}=b_{g_{2}g_{2}}^{f}=1\text{ and }%
b_{g_{1}g_{2}}^{f}=2\text{,} \\ 
b_{g_{1}g_{1}}^{g}=b_{g_{2}g_{2}}^{g}=1\text{ and }%
b_{g_{1}g_{2}}^{g}=b_{g_{2}g_{1}}^{g}=\frac{3}{2}\text{.}%
\end{array}%
\right.  \label{ExampleWeights1}
\end{equation}%
We conclude from Lemma\ (\ref{NormLemma}) and Equalities (\ref%
{ExampleWeights1}) that:%
\begin{equation*}
\left\{ 
\begin{array}{l}
\left\Vert W_{g_{1}}^{f}\right\Vert =\left\Vert W_{g_{2}}^{f}\right\Vert
=\left\Vert W_{g_{1}}^{g}\right\Vert =\left\Vert W_{g_{2}}^{g}\right\Vert =1%
\text{,} \\ 
\vspace{-0.1in} \\ 
\left\Vert W_{g_{1}}^{f}W_{g_{1}}^{f}\right\Vert =\left\Vert
W_{g_{2}}^{f}W_{g_{1}}^{f}\right\Vert =\left\Vert
W_{g_{2}}^{f}W_{g_{2}}^{f}\right\Vert =1\text{ and }\left\Vert
W_{g_{1}}^{f}W_{g_{2}}^{f}\right\Vert =\frac{1}{\sqrt{2}}\text{,} \\ 
\left\Vert W_{g_{1}}^{g}W_{g_{1}}^{g}\right\Vert =\left\Vert
W_{g_{2}}^{g}W_{g_{2}}^{g}\right\Vert =1\text{ and }\left\Vert
W_{g_{2}g_{1}}^{g}\right\Vert =\left\Vert W_{g_{1}g_{2}}^{g}\right\Vert =%
\sqrt{\frac{2}{3}}\text{.}%
\end{array}%
\right.
\end{equation*}%
Assume now that there exists an $\Phi $ isomorphism in \textrm{NCD} from $%
\mathcal{A}\left( \mathcal{D}_{f}\right) $ onto $\mathcal{A}\left( \mathcal{D%
}_{g}\right) $. By Thullen, we have $\widehat{\Phi }_{1}(0)=0$. Now, by
Theorem (\ref{Linear}), there exist $t_{11},t_{12},t_{21},t_{22}\in \mathbb{C%
}$ such that:%
\begin{equation*}
\left\{ 
\begin{array}{l}
\Phi (W_{g_{1}}^{f})=t_{11}W_{g_{1}}^{g}+t_{12}W_{g_{2}}^{g}\text{,} \\ 
\Phi (W_{g_{2}}^{f})=t_{21}W_{g_{1}}^{g}+t_{22}W_{g_{2}}^{g}\text{.}%
\end{array}%
\right.
\end{equation*}%
Moreover, Lemma (\ref{NormLemma2}) implies that $T=\left[ 
\begin{array}{cc}
t_{11} & t_{12} \\ 
t_{21} & t_{22}%
\end{array}%
\right] $ is unitary. Now:%
\begin{eqnarray*}
\Phi (W_{g_{1}}^{f}W_{g_{1}}^{f}) &=&\left(
t_{11}W_{g_{1}}^{g}+t_{12}W_{g_{2}}^{g}\right) \left(
t_{11}W_{g_{1}}^{g}+t_{12}W_{g_{2}}^{g}\right) \\
&=&t_{11}^{2}W_{g_{1}g_{1}}^{g}+t_{11}t_{12}\left(
W_{g_{1}g_{2}}^{g}+W_{g_{2}g_{1}}^{g}\right) +t_{12}^{2}W_{g_{2}g_{2}}^{g}%
\text{.}
\end{eqnarray*}%
From Lemma (\ref{NormLemma2}) and since $\Phi $ is an isometry we conclude
that:%
\begin{equation*}
1=\left\Vert W_{g_{1}}^{f}W_{g_{1}}^{f}\right\Vert =\left\Vert \Phi \left(
W_{g_{1}}^{f}W_{g_{1}}^{f}\right) \right\Vert =\sqrt{\left\vert
t_{11}\right\vert ^{2}+\frac{4}{3}\left\vert t_{11}t_{12}\right\vert
^{2}+\left\vert t_{12}\right\vert ^{2}}\text{.}
\end{equation*}%
Since $\left\vert t_{11}\right\vert ^{2}+\left\vert t_{12}\right\vert ^{2}=1$
we conclude that $\left\vert t_{11}t_{12}\right\vert =0$. Hence either $%
t_{11}=0$ or $t_{12}=0$.

If $t_{11}=0$ and since $T$ is unitary, we conclude that then $\left\vert
t_{12}\right\vert =\left\vert t_{21}\right\vert =1$ and $\left\vert
t_{22}\right\vert =0$. We then have:%
\begin{equation*}
1=\left\Vert W_{g_{2}g_{1}}^{f}\right\Vert =\left\Vert \Phi \left(
W_{g_{2}g_{1}}^{f}\right) \right\Vert =\left\Vert
W_{g_{1}g_{2}}^{g}\right\Vert =\sqrt{\frac{2}{3}}
\end{equation*}%
which is obviously a contradiction! A similar computation shows that if
instead $t_{12}=0$ then we are led to a similar contradiction. Hence, $\Phi $
does not exist and our theorem is proven.
\end{proof}

\bigskip We wish to point out that this proof is an example where, while $%
\mathbb{D}^{1}$ was not able to distinguish between these two examples,
results from multivariate complex analysis together with our invariant $%
\mathbb{D}^{k}$ allows us to prove that they are not isomorphic. We chose to
use Thullen's result in our proof, yet other routes would have been
possible. For instance, we could have used the results of \cite{Sunada} and 
\cite{fu-isaev-krantz} as well. Since this example only requires $\mathbb{D}%
^{1}$ and $\mathbb{D}^{2}$, Thullen's classification is sufficient. Yet, for
many examples, one may need to use $\mathbb{D}^{k}$ for $k\geq 2$, in which
case \cite{Sunada} and \cite{fu-isaev-krantz} are appropriate.

\subsection{Characterization of Noncommutative Disk Algebras}

In this section we characterize the noncommutative disk algebras $\mathcal{A}%
\left( \mathcal{D}_{f}\right) $ in terms of their symbol $f$. We also note
that the domains of $\mathbb{C}^{n}$ obtained from our algebras appear to be
new and interesting examples.

The following simple result will prove useful:

\begin{lemma}
\label{Scaling}Let $f=\sum_{\alpha \in \mathbb{F}_{n}^{+}}a_{\alpha
}X_{\alpha }$ be a positive regular free series. Let $W_{1},\ldots ,W_{n}$
be the weighted shifts associated with $f$. If $c_{1},\ldots ,c_{n}$ are
non-negative numbers, then $c_{1}W_{1},c_{2}W_{2},\ldots ,c_{n}W_{n}$ are
the weighted shifts spaces associated to some positive regular free series $%
g=\sum_{\alpha \in \mathbb{F}_{n}^{+}}a_{\alpha }^{\prime }X_{\alpha }$.
Therefore $\mathcal{A}_{n}\left( \mathcal{D}_{f}\right) \cong \mathcal{A}%
_{n}(\mathcal{D}_{g}),$ and $\mathcal{D}_{g}\left( 
\mathbb{C}
\right) $ is obtained from $\mathcal{D}_{f}\left( 
\mathbb{C}
\right) $ by scaling the coordinates in $%
\mathbb{C}
^{n}$ (i.e., by a diagonal linear map with respect to the canonical basis of 
$%
\mathbb{C}
^{n}$).
\end{lemma}

\begin{proof}
It is clearly enough to scale only one of the $W_{i}$'s. Assume that $%
W_{1}^{\prime }=c_{1}W_{1},$ and that $W_{j}^{\prime }=W_{j}$ for $j\in
\left\{ 2,\ldots ,n\right\} $. We claim that there exist scalars $\left(
a_{\alpha }^{\prime }\right) _{\left\vert \alpha \right\vert \geq 1}$
satisfying the conditions (\ref{Condition1}) such that $W_{1}^{\prime
},W_{2}^{\prime },\ldots ,W_{n}^{\prime }$ are the weighted shifts
associated to the positive regular free series $g=\sum_{\alpha \in \mathbb{F}%
_{n}^{+}}a_{\alpha }^{\prime }X_{\alpha }$.

Let $r_{1}:\mathbb{F}_{n}^{+}\rightarrow 
\mathbb{N}
$ be the function that counts the number of $g_{1}$'s in a word of $\mathbb{F%
}_{n}^{+}.$ For example, $r_{1}\left( g_{2}g_{1}g_{3}g_{1}\right) =2$ and $%
r_{1}\left( g_{2}g_{4}\right) =0$. For $\alpha \in \mathbb{F}_{n}^{+},$
define%
\begin{equation*}
a_{\alpha }^{\prime }=\frac{1}{c_{1}^{2r_{1}\left( \alpha \right) }}%
a_{\alpha }.
\end{equation*}%
Notice that $\left( a_{\alpha }^{\prime }\right) _{\left\vert \alpha
\right\vert \geq 1}$ satisfies the conditions (\ref{Condition1}). Then from
Theorem \ref{Popescu2} there exist positive real numbers $b_{\alpha
}^{\prime }$ and $n$ weighted shifts $V_{1},\ldots ,V_{n}$ associated to $%
g=\sum_{\alpha \in \mathbb{F}_{n}^{+}}a_{\alpha }^{\prime }X_{\alpha }$. The 
$b_{\alpha }^{\prime }$'s must satisfy equation $(\ref{popescu2-2})$ and we
can easily check that these are given by 
\begin{equation*}
b_{\alpha }^{\prime }=\frac{1}{c_{1}^{2r_{1}\left( \alpha \right) }}%
b_{\alpha }.
\end{equation*}%
Therefore, the weights derived from $\left( a_{\alpha }^{\prime }\right)
_{\left\vert \alpha \right\vert \geq 1}$ are given by the formulas:%
\begin{eqnarray*}
V_{1}\delta _{\alpha } &=&\sqrt{\frac{b_{\alpha }^{\prime }}{b_{g_{1}\alpha
}^{\prime }}}\delta _{g_{1}\alpha }=c_{1}\sqrt{\frac{b_{\alpha }}{%
b_{g_{1}\alpha }}}\delta _{g_{1}\alpha }=c_{1}W_{1}\delta _{a} \\
V_{j}\delta _{\alpha } &=&\sqrt{\frac{b_{\alpha }^{\prime }}{b_{g_{j}\alpha
}^{\prime }}}\delta _{g_{j}\alpha }=W_{j}\delta _{a}\text{ }\qquad \text{for 
}j\geq 2,
\end{eqnarray*}%
and the proof is complete.

We will now see that $\mathcal{D}_{g}\left( 
\mathbb{C}
\right) $ can be obtained from $\mathcal{D}_{f}\left( 
\mathbb{C}
\right) $ by scaling the canonical coordinates. For $\alpha \in \mathbb{F}%
_{n}^{+},$ notice that 
\begin{equation*}
\left\vert \lambda _{\alpha }\right\vert ^{2}=\left\vert \lambda
_{1}\right\vert ^{2r_{1}\left( \alpha \right) }\left\vert \lambda
_{2}\right\vert ^{2r_{2}\left( \alpha \right) }\cdots \left\vert \lambda
_{n}\right\vert ^{2r_{n}\left( \alpha \right) },
\end{equation*}%
where for $i\leq n,$ $r_{i}\left( \alpha \right) $ is the number of $g_{i}$%
's in the word $\alpha .$ Then we check easily that $a_{a}^{\prime
}\left\vert \lambda _{\alpha }\right\vert ^{2}=a_{\alpha }\left\vert \lambda
_{\alpha }^{\prime }\right\vert ,$ where 
\begin{equation*}
\lambda ^{\prime }=\left( \frac{\lambda _{1}}{c_{1}},\lambda _{2},\ldots
,\lambda _{n}\right) \in 
\mathbb{C}
^{n}.
\end{equation*}%
Then it follows that $\sum_{\left\vert \alpha \right\vert \geq 1}a_{\alpha
}^{\prime }\left\vert \lambda _{\alpha }\right\vert ^{2}\leq 1$ iff $%
\sum_{\left\vert \alpha \right\vert \geq 1}a_{\alpha }\left\vert \lambda
_{\alpha }^{\prime }\right\vert ^{2}\leq 1,$ which proves the result.
\end{proof}

\bigskip

\bigskip

Let $f,g$ be positive regular free series, and assume that there exists an
isomorphism $\Phi :\mathcal{A}\left( \mathcal{D}_{f}\right) \longrightarrow 
\mathcal{A}\left( \mathcal{D}_{g}\right) $ in \textrm{NCD }such that $%
\widehat{\Phi }_{1}(0)=0.$ After rescailing we can assume that for $i\leq
n,a_{i}^{f}=a_{i}^{g}=1$ and $b_{i}^{f}=b_{i}^{g}=1.$ Then by Lemma \ref%
{NormLemma2}, $\left\Vert \sum_{i=1}^{n}c_{i}W_{i}^{f}\right\Vert
=\left\Vert \sum_{i=1}^{n}c_{i}W_{i}^{g}\right\Vert =\sqrt{%
\sum_{i=1}^{n}\left\vert c_{i}\right\vert ^{2}}$ for all $c_{1},\ldots
,c_{n}\in 
\mathbb{C}
.$ Since $\sqrt{\sum_{i=1}^{n}\left\vert c_{i}\right\vert ^{2}}=\left\Vert
\sum_{i=1}^{n}c_{i}W_{i}^{f}\right\Vert =\left\Vert \Phi \left(
\sum_{i=1}^{n}c_{i}W_{i}^{f}\right) \right\Vert =$ $\left\Vert
\sum_{i=1}^{n}c_{i}\sum_{j=1}^{n}m_{ij}W_{j}^{g}\right\Vert =\left\Vert
\sum_{j=1}^{n}\left( \sum_{i=1}^{n}c_{i}m_{ij}\right) W_{j}^{g}\right\Vert =%
\sqrt{\sum_{j=1}^{m}\left\vert \sum_{i=1}^{n}c_{i}m_{ij}\right\vert ^{2}},$
we conclude that $M$ is unitary.

\bigskip

\begin{corollary}
\label{unitary-corollary}After rescaling, we can assume that the $M$ of
Theorem \ref{Linear}\ is unitary.
\end{corollary}

We denote by $S_{1},\ldots ,S_{n}$ the unilateral shifts on $\ell ^{2}\left( 
\mathbb{F}_{n}^{+}\right) $ generating $\mathcal{A}_{n}$ as in section 2. We
need the following:

\begin{theorem}[Davidson-Pitts \protect\cite{Davidson-Pitts-automorphism}]
\label{davidson-pitts}For every $\chi \in \mathbb{D}^{1}\left( \mathcal{A}%
_{n}\right) $ (the open unit ball of $%
\mathbb{C}
^{n}$) there exists an automorphism $\Phi :\mathcal{A}_{n}\rightarrow 
\mathcal{A}_{n}$ in \textrm{NCD }such that $\widehat{\Phi }_{1}(\chi )=0.$
\end{theorem}

We are now ready to characterize noncommutative disk algebras by their
symbol.

\begin{theorem}
Let $f$ be a positive regular $n$-free series. Then $\mathcal{A}\left( 
\mathcal{D}_{f}\right) $ is isomorphic to $\mathcal{A}_{n}$ in \textrm{NCD }%
if and only if $f=\sum_{i=1}^{n}c_{i}X_{i}$ with $\left( c_{1},\ldots
,c_{n}\right) \in \mathbb{C}^{n}$.
\end{theorem}

\begin{proof}
By Lemma (\ref{Scaling}) we can assume, without loss of generality, that $%
a_{g_{1}}=\cdots =a_{g_{n}}=1$. Let $\Phi $ be an isomorphism from $\mathcal{%
A}_{n}$ onto $\mathcal{A}\left( \mathcal{D}_{f}\right) $ in \textrm{NCD}.
Let $\chi =\widehat{\Phi }_{1}(0)$, let $\varphi $ be an automorphism of $%
\mathcal{A}_{n}$ such that $\widehat{\varphi }_{1}\left( \chi \right) =0$
given by Theorem \ref{Linear}. We set $\eta =\Phi \circ \varphi $. By
functoriality, $\widehat{\eta }_{1}\left( 0\right) =0$ and thus $\eta $
satisfies the hypotheses of Theorem (\ref{Linear}). Therefore, there exists
a matrix $U=\left( u_{ij}\right) _{i,j\in \left\{ 1,\ldots ,n\right\} }\in
M_{n\times n}$ such that:%
\begin{equation*}
\eta (S_{i})=\sum_{j=1}^{n}u_{ij}W_{g_{j}}^{f}\text{.}
\end{equation*}%
As in Corollary (\ref{unitary-corollary}), the matrix $U$ is unitary. We
compute:%
\begin{eqnarray*}
1 &=&\left\Vert S_{i}S_{j}\right\Vert =\left\Vert \Phi \left(
S_{i}S_{j}\right) \right\Vert \\
&=&\left\Vert
\sum_{k,l=1}^{n}u_{ik}u_{jl}W_{g_{k}}^{f}W_{g_{l}}^{f}\right\Vert =\sqrt[2]{%
\sum_{k,j=1}^{n}\left\vert u_{ik}u_{jl}\right\vert ^{2}\frac{1}{b_{kl}}}%
\text{.}
\end{eqnarray*}%
To simplify notation in this proof, we write $a_{i_{1}\cdots i_{k}}$ for $%
a_{g_{i_{1}}\cdots g_{i_{k}}}$ and similarly $b_{i_{1}\cdots i_{k}}$ for $%
b_{g_{i_{1}}\cdots g_{i_{k}}}$. By Identity (\ref{popescu2-5}), we have $%
b_{kl}=b_{k}a_{l}+a_{kl}=1+a_{kl}\geq 1$ and notice that:%
\begin{equation*}
\sqrt[2]{\sum_{k,l=1}^{n}\left\vert u_{ik}u_{jl}\right\vert ^{2}}=1\text{.}
\end{equation*}

Fix $k_{0},l_{0}\leq n$. Let $i,j\leq n$ such that $u_{ik_{0}}\not=0$ and $%
u_{jl_{0}}\not=0$. Then since $\sum_{k,j=1}^{n}\left\vert
u_{ik}u_{jl}\right\vert ^{2}\frac{1}{b_{kl}}=1$ and $\sum_{k,j=1}^{n}\left%
\vert u_{ik}u_{jl}\right\vert ^{2}=1$, and $u_{ik_{0}}u_{jl_{0}}\not=0$ by
assumption, we must have $b_{k_{0}l_{0}}=1$ (otherwise, since $b_{kl}>1$ for
all $k,l$, both sum of squares could not be one at the same time). Hence $%
a_{k_{0}l_{0}}=0$. As $k_{0},l_{0}$ are arbitrary, we conclude that $%
a_{\alpha }=0$ for all $\alpha \in \mathbb{F}_{n}^{+}$ with $\left\vert
\alpha \right\vert =2$.

We proceed identically by induction.
\end{proof}

\subsection{Examples of Domains in $%
\mathbb{C}
^{nk^{2}}$ with non compact automorphic group}

The Riemann mapping theorem tells us that a bounded simply connected domain $%
D\subset $ $%
\mathbb{C}
$ is homogeneous, (i.e., if $z\in D$ is fixed, for every $w\in D$ there
exists a biholomorphic map from $D$ to $D$ that maps $z$ to $w$), and hence
the group of automorphisms of $D$ is non compact. The situation is much more
rigid in higher dimensions. In 1927, Kritikos \cite{kritikos} proved that
every automorphism of $\left\{ \left( z_{1},z_{2}\right) \in 
\mathbb{C}
^{2}:\left\vert z_{1}\right\vert +\left\vert z_{2}\right\vert <1\right\} $
fixes the origin and is linear. In 1931, Thullen \cite{Thullen}
characterized all bounded proper Reinhardt domains of $%
\mathbb{C}
^{2}$ that contain the origin and that have non compact automorphic group
(i.e., there exists a biholomorphic map that does not fix the origin).
Domains of $%
\mathbb{C}
^{n}$ with non-compact automorphism group have been studied extensively,
leading in particular to many deep characterizations of the unit ball among
domains in $%
\mathbb{C}
^{n},$ as well as many families of domains of $%
\mathbb{C}
^{n}$ with interesting geometries. We refer to \cite{isaev-krantz} for a
survey and references.

\bigskip

We conclude this paper with a remark that Theorem \ref{Functor} combined
with the result of Davidson and Pitts \cite{Davidson-Pitts-automorphism}
mentioned in Theorem \ref{davidson-pitts} can be used to obtain new bounded
domains in $%
\mathbb{C}
^{nk^{2}}$ with non-compact automorphic group.

\begin{proposition}
For every $n,k\in 
\mathbb{N}
,$ the domain 
\begin{equation*}
\mathbb{D}^{k}\left( \mathcal{A}_{n}\right) =\left\{ \left( T_{1},\ldots
,T_{n}\right) \in \left( M_{k\times k}\right) ^{n}:T_{1}T_{1}^{\ast }+\cdots
+T_{n}T_{n}^{\ast }\leq I\right\} ^{\circ }\subset 
\mathbb{C}
^{nk^{2}}
\end{equation*}%
does not have a compact automorphic group.
\end{proposition}

\begin{proof}
By \cite{Davidson-Pitts-automorphism}, for any $\chi \in \mathbb{D}%
^{1}\left( \mathcal{A}_{n}\right) $ there exists an automorphism $\Phi $ of $%
\mathcal{A}_{n}$ such that $\widehat{\Phi _{1}}(0)=\chi $. Now, using the
expression of $\widehat{\Phi _{k}}$ in term of the expansion in series of $%
\Phi (W_{1}),\ldots ,\Phi (W_{n})$ given by Lemma\ (\ref{radialCv}), we
conclude easily that $\widehat{\Phi _{k}}\left( 0\right) =\left( \chi
_{1}1_{k},\ldots ,\chi _{n}1_{k}\right) $. Hence, the orbit of the point $%
0\in \mathbb{D}^{k}\left( \mathcal{A}_{n}\right) $ by the action of the
automorphism group of $\mathcal{A}_{n}$ admits a boundary point of $\mathbb{D%
}^{k}\left( \mathcal{A}_{n}\right) $ as a limit. Therefore, the orbit of $0$
by the automorphism group of $\mathbb{D}^{k}\left( \mathcal{A}_{n}\right) $
is not compact. Hence the group of automorphism of $\mathbb{D}^{k}\left( 
\mathcal{A}_{n}\right) $ is not compact, since it acts strongly continuously
on $\mathbb{D}^{k}\left( \mathcal{A}_{n}\right) $.
\end{proof}

If $n=k=2$, $\mathbb{D}^{2}\left( \mathcal{A}_{2}\right) $ consists of the
interior of the set of the $\lambda \in 
\mathbb{C}
^{8}$ such that if $T_{1}=\left[ 
\begin{array}{cc}
\lambda _{1} & \lambda _{2} \\ 
\lambda _{5} & \lambda _{6}%
\end{array}%
\right] $ and $T_{2}=\left[ 
\begin{array}{cc}
\lambda _{3} & \lambda _{4} \\ 
\lambda _{7} & \lambda _{8}%
\end{array}%
\right] ,$ then $T_{1}T_{1}^{\ast }+T_{2}T_{2}^{\ast }\leq I.$ Notice that 
\begin{equation*}
T_{1}T_{1}^{\ast }+T_{2}T_{2}^{\ast }=\left[ 
\begin{array}{cc}
\left\vert \lambda _{1}\right\vert ^{2}+\left\vert \lambda _{2}\right\vert
^{2} & \lambda _{1}\overline{\lambda _{5}}+\lambda _{2}\overline{\lambda _{6}%
} \\ 
\lambda _{5}\overline{\lambda _{1}}+\lambda _{6}\overline{\lambda _{2}} & 
\left\vert \lambda _{5}\right\vert ^{2}+\left\vert \lambda _{6}\right\vert
^{2}%
\end{array}%
\right] +\left[ 
\begin{array}{cc}
\left\vert \lambda _{3}\right\vert ^{2}+\left\vert \lambda _{4}\right\vert
^{2} & \lambda _{3}\overline{\lambda _{7}}+\lambda _{4}\overline{\lambda _{8}%
} \\ 
\lambda _{7}\overline{\lambda _{3}}+\lambda _{8}\overline{\lambda _{4}} & 
\left\vert \lambda _{7}\right\vert ^{2}+\left\vert \lambda _{8}\right\vert
^{2}%
\end{array}%
\right] \text{.}
\end{equation*}%
Now, $T_{1}T_{1}^{\ast }+T_{2}T\leq 1$ if and only if $\det \left(
1-T_{1}T_{1}^{\ast }-T_{2}T_{2}^{\ast }\right) \geq 0$ and the $\left(
1,1\right) $ entry of $1-T_{1}T_{1}^{\ast }-T_{2}T_{2}^{\ast }$ is
nonnegative. Hence $\mathbb{D}^{2}\left( \mathcal{A}_{2}\right) $ consists
of the interior of the set of the $\lambda \in 
\mathbb{C}
^{8}$ such that:%
\begin{equation*}
\begin{array}{l}
\left\vert \lambda _{1}\right\vert ^{2}+\left\vert \lambda _{2}\right\vert
^{2}+\left\vert \lambda _{3}\right\vert ^{2}+\left\vert \lambda
_{4}\right\vert ^{2}\leq 1\text{, and} \\ 
\vspace{-0.1in} \\ 
\sum_{i=1}^{8}\left\vert \lambda _{i}\right\vert
^{2}-\sum_{i=1}^{7}\sum_{j=i+1}^{8}\left\vert \lambda _{i}\lambda
_{j}\right\vert ^{2}+\sum_{i=1}^{3}\sum_{j=i+1}^{4}\left\vert \lambda _{i}%
\overline{\lambda _{j}}+\overline{\lambda _{i+4}}\lambda _{j+4}\right\vert
^{2}\leq 1\text{.}%
\end{array}%
\end{equation*}%
This domain has a noncompact automorphism group but appears not to fit in
the families of domains in $\mathbb{C}^{n}$ with non compact automorphic
group encountered in the literature on several complex variables. Thus,
distinguishing noncommutative domain algebras may also involve the study of
new interesting domains and their biholomorphic equivalence.

\end{document}